\documentclass[leqno,a4paper]{amsart}

\usepackage{amsmath,amsfonts,amsthm,amssymb,dsfont}

\usepackage{times}
\usepackage{microtype}

\usepackage{cite}

\usepackage{mathrsfs}
\usepackage{enumerate, xspace}
\usepackage{graphicx}
\usepackage{color}

    \usepackage[all, knot]{xy}
    \xyoption{arc} 

\usepackage{pgf,tikz}\usepackage{mathrsfs}\usetikzlibrary{arrows}

\usepackage{verbatim}

\usepackage{caption}
\usepackage{subcaption}

\usepackage[pdfauthor={Thomas Delzant and Markus Steenbock},pdftitle={Product set growth in groups and hyperbolic geometry},pdfkeywords={product sets, small tripling,  growth, hyperbolic groups, hyperbolic space, acylindrical actions, small cancellation theory},pdftex]{hyperref}

\hypersetup{colorlinks,%
citecolor=black,%
filecolor=black,%
linkcolor=black,%
urlcolor=black,%
}

\theoremstyle{plain}

				\newtheorem{thm}{Theorem}[section]
				\newtheorem{prop}[thm]{Proposition}
				\newtheorem{lem}[thm]{Lemma}		
						\newtheorem{cor}[thm]{Corollary}

\theoremstyle{definition}		\newtheorem{df}[thm]{Definition}
						\newtheorem{ex}[thm]{Example}	
				
\newtheorem{rem}[thm]{Remark}

	\theoremstyle{remark} 	

\DeclareMathOperator{\diam}{diam}

\DeclareMathOperator{\length}{length}

\title{Product Set Growth in Groups and Hyperbolic Geometry}
\date{}
\author{Thomas Delzant}
\address{IRMA\\ Universit\'e de Strasbourg et CNRS\\ 67000 Strasbourg \\ France}
\email{delzant@math.unistra.fr}
\author{Markus Steenbock }
\address{IRMAR\\ Univ Rennes et CNRS \\  35000 Rennes \\ France}
\email{markus.steenbock@univ-rennes1.fr}

\subjclass[2010]{20F65; 20F67, 20E08, 20F70}
\keywords{product sets, small tripling, growth, hyperbolic groups, acylindrical actions, small cancellation}

\begin{document}

\begin{abstract}
Generalising results of Razborov and Safin, and answering a question of Button, 
we prove that for every hyperbolic group there exists a constant $\alpha >0$ such that for every finite subset $U$ that is not contained in a virtually cyclic subgroup $|U^n|\geqslant (\alpha |U|)^{[(n+1)/2]}$. 
Similar estimates are established for groups acting acylindrically on trees or hyperbolic spaces. 
\end{abstract}

 \maketitle

\smallskip

\section{Introduction}

  If $U$ is a subset of a group $G$, then $U^n$ denotes the subset of $G$ of all products of $n$ elements of $U$. 
 The study of the growth of $U^n$ is motivated by the theory of approximate subgroups in the sense of Tao, see \cite{terence_product_2008}. Approximate subgroups of abelian, nilpotent or solvable groups have been extensively studied in additive and arithmetic combinatorics, for example in \cite{tao_freiman_2010,breuillard_structure_2012,button_explicit_2013}. 
  
We study the growth of  $U^n$ in the framework of groups acting on hyperbolic spaces, for instance hyperbolic groups and groups acting on trees. This extends previous results for $SL_2(\mathbb{Z})$, free groups, free products and surface groups \cite{chang_product_2008,razborov_product_2014,safin_powers_2011,button_explicit_2013}.

  \subsection{Hyperbolic groups}  A discrete group $G$ is hyperbolic if it admits a proper and cocompact action by isometries on a $\delta$--hyperbolic geodesic proper space $X$ \cite{gromov_hyperbolic_1987}. The cardinality of $U$ is denoted by $|U|$.

\begin{thm}\label{IT: hyperbolic groups} 
For every hyperbolic group $G$ there is a constant $\alpha>0$ such that for every finite subset $U\subset G$ that is not contained in a virtually cyclic subgroup $|U^3|\geqslant (\alpha |U|)^2.$ 
\\
More generally,  for all natural numbers $n$, $|U^n|\geqslant (\alpha |U|)^{[(n+1)/2]}$.
\end{thm}
Here $[(n+1)/2]$ is the integral part of $(n+1)/2$. Theorem \ref{IT: hyperbolic groups} answers a question in \cite{button_explicit_2013}. In contrast, if $U$ is a finite subset of $\mathbb{Z}$ (respectively a nilpotent group), the growth of $U^n$ is linear (respectively polynomial) in $n$. 

\begin{rem}\label{IR: breuillard fujiwara} Let $h(U):= \frac{1}{n} \log |U^n|$ be the algebraic entropy of $U$. By Theorem \ref{IT: hyperbolic groups}, for all $U$ in a hyperbolic group that do not generate a virtually cyclic subgroup $h(U)> \frac{1}{2}\log \left( \alpha |U|\right)$. Note that every hyperbolic group has uniform positive lower bound on the algebraic entropy of a finite generating set (i.e. uniform exponential growth), cf. \cite{koubi_croissance_1998,arzhantseva_lower_2006,besson_cuvature_2017,breuillard_joint_2018}.
\end{rem}

A natural question is to understand the meaning of $\alpha$ in geometric terms, such as, the diameter of the thin part and the injectivity radius of $X/G$, for convex cocompact Kleinian groups for example. If $G$ is hyperbolic, then $\alpha$ can be explicitly estimated in terms of the cardinality of the (closed) ball of radius $10\delta$ in a Cayley graph of $G$, where $\delta$ is the hyperbolicity constant from the four point definition of hyperbolicity (Definition \ref{D: hyperbolicity} below), see Theorem \ref{IT: hyperbolic groups}.

\begin{ex} Let  $F$ be  a finite group, $H$  a free group and $G=F\times H$. If $U_1\subset H$ and $U:= F\times U_1$, then $U^n=F\times U_1^n$ and $\alpha$ depends on the size of $F$, which will be elliptic on every space $G$ acts.  
\end{ex}

By the following example of \cite{safin_powers_2011}, the exponent $[(n+1)/2]$, see Theorem \ref{IT: hyperbolic groups}, is optimal.

 \begin{ex} Let $g$ and $h$ generate a free subgroup and let $U_N:=\{g^{-N},\, \ldots,\, g^{-1},\, 1,\, g,\, \ldots, \, g^N,\, h\}$.  Then $|U_N^{n}| \sim N^{[(n+1)/2]}$, while $|U_N|=2N+1$.
\end{ex}

\subsection{Energy and acylindricity}

To analyse growth of products of a subset $U$ of a group acting on a $\delta$--hyperbolic space $X$, we consider the \emph{normalised $\ell^1$--energy} 
$$E(U):=\inf_{x\in X} \frac{1}{|U|} \sum_{u\in U}|ux-x|$$
 and fix a point $x_0$ where this infimum is achieved (up to an error of  $\delta$). 
 
 The \emph{displacement of $U$} is defined by $\lambda_0(U):=\max_{u\in U} |ux_0-x_0
 |$.

\begin{rem}
  If $G$ acts on a $\delta$--hyperbolic space and $U$ is contained in an elliptic subgroup, then $E(U)\leqslant 10 \delta$, see Proposition \ref{P: Energy and displacement}. For instance, if $\delta=0$ then $U$ is fixes a point and $\lambda_0(U)=0$.
\end{rem}

\begin{ex} If $X/G$ is a manifold (orbifold) of constant curvature $-1$ and $\lambda_0(U)$ is smaller than the Margulis constant, then $U$ generates a  cyclic (virtually cyclic) subgroup. 
\end{ex}
 
 \begin{df}[\hspace{-0.15mm}{\cite{sela_acylindricity_1997,bowditch_tight_2008}}] \label{ID: acylindrical}
The action of $G$ on a $\delta$--hyperbolic space $X$ is  \emph{acylindrical}, or more precisely \emph{$(\kappa_0,N_0)$--acylindrical}, if there is $\kappa_0\geqslant \delta$ and $N_0>0$ such that for all $x$ and $y$ in $X$ that are at least $\kappa_0$ apart, there are at most $N_0$ isometries $g\in G$ such that $|gx-x| \leqslant 100 \delta$ and $|gy- y|\leqslant 100 \delta$.
\end{df}

\begin{ex}\label{IE: Riemannian manifold}
In the case of the fundamental group of a compact Riemannian manifold of constant negative curvature $ -a^2$, the hyperbolicity constant is (of the order of) $1/a$. If $l$ denotes the injectivity radius and is smaller than the Margulis constant, then  the acylindricity constant $\kappa_0$ is about $1/l$, so that a bound on the acylindricity constant yields a lower bound on the injectivity radius. 
 If one considers fundamental groups of orbifolds, however, the size of the elliptic subgroups (this is the meaning of $N_0$), that can be arbitrarily large, must be taken into account in growth estimates. 
\end{ex}

\begin{ex}\label{IE: groups acting on trees} For groups acting on trees, typically an amalgamated product $A*_CB$ acting on its Bass-Serre tree, the acylindricity constant can be expressed in algebraic terms. If $C$ is finite, the action is $(1,|C|)$--acylindrical. If $C$ is  malnormal in $A$ (that is, $aCa^{-1}\cap C=\{1\}$ unless $a\in C$), the action is $(2,1)$--acylindrical. Groups acting acylindrically on trees are discussed, for instance, in \cite{sela_acylindricity_1997,delzant_access_1999}.   
\end{ex}

\begin{ex} Most mapping class groups act acylindrically on curve graphs \cite{bowditch_tight_2008}.
\end{ex}

\subsection{Groups acting on trees}

Recall that a metric tree is a  $0$--hyperbolic space. Let $X$ be a simplicial tree endowed with the geodesic metric such that every  edge has length $\rho_0$. The action of $G$ on $X$ is $\kappa$--acylindrical (in the sense of Sela \cite{sela_acylindricity_1997}) if it is $(\kappa,1)$--acylindrical in the sense of Definition \ref{ID: acylindrical} above, namely, if the pointwise stabiliser of a segment of length $\kappa$ is trivial. Note that  $\rho_0$ as well as $\kappa$ are lengths. If $\kappa< \rho_0$,  then $G$ is a free product of groups, studied in \cite{button_explicit_2013}.

\begin{thm}\label{IT: acylindrical trees}  For every $\rho_0>0$ and $\kappa$ there is a constant $\alpha>0$ such that for every group $G$ with a $\kappa$--acylindrical action on a simplicial tree of edge length $\rho_0$  
and for every finite $U\subset G$  of displacement $\lambda_0(U)\geqslant 10^{14} \kappa$  
 that is not contained in an infinite virtually cyclic  
subgroup 
 $|U^3|\geqslant ( \alpha  |U|)^2$.  \\
More generally,  for all natural numbers $n$, $|U^n|\geqslant (\alpha |U|)^{[(n+1)/2]}$.
\end{thm}

\begin{rem} \label{IR: acylindrical tree} We may take $\alpha= \frac{1}{10^{15}}\frac{\rho_0^2}{\kappa^2}$. Note that $\alpha$ is unit free and does not depend on the group, for instance, this gives uniform estimates for all acylindrical amalgamated products $A*_CB$. 
\end{rem}

Note that Theorem \ref{IT: acylindrical trees} extends the main result of \cite{button_explicit_2013}:

\begin{rem} In the case of a free product acting on its Bass-Serre tree, we can take $\kappa=\rho_0/10^{14}$. Then, if $U$ is not conjugated into one of the vertex stabilisers, $\lambda_0(U)\geqslant 10^{14}\kappa$. By Theorem \ref{IT: acylindrical trees}, for every such $U$ that is not in a virtually cyclic subgroup $|U^n|\geqslant (\alpha |U|)^{[(n+1)/2]}$.
\end{rem}

To underline the importance of the acylindricity hypothesis, we recall that in the case of Baumslag-Solitar groups, studied by Button  \cite[Section 4, p. 76]{button_explicit_2013}, there is no (uniform) $\alpha$ such that for all $U$ that do not generate virtually nilpotent subgroups  $|U^3|\geqslant ( \alpha  |U|)^2$ .

\subsection{Groups acting on hyperbolic spaces}

Finally, we assume that $\delta >0$, and discuss groups acting acylindrically on $\delta$--hyperbolic spaces.   

\begin{thm}\label{IT: acylindrical hyperbolic}
For every $\delta> 0$, $\kappa_0\geqslant \delta$, $N_0>0$ there is a constant $\alpha >0$ such that for every group $G$ that acts $(\kappa_0,N_0)$--acylindrically on a $\delta$--hyperbolic space $X$ and for every finite $U\subset G$ of displacement $\lambda_0(U)\geqslant 10^{14}\kappa_0 \log_2(2|U|)$ that is not contained in an infinite virtually cyclic subgroup 
 $$|U^3|\geqslant \left( \frac{\alpha}{\log_2^6(2|U|)}  |U| \right)^2.$$
More generally,  for all natural numbers $n$,  $$|U^n|\geqslant \left(\frac{\alpha}{\log_2^6(2|U|)} |U|\right)^{[(n+1)/2]}.$$
\end{thm}

\begin{rem}\label{IR: acylindrical 1} 
 We may take $\alpha = \frac{1}{10^{50}N_0^{6}} \cdot \frac{\delta^2 }{\kappa_0^2}$. 
 To compare with Theorem \ref{IT: acylindrical trees} and Remark \ref{IR: acylindrical tree} note that by acylindricity the injectivity radius of the action is $\rho_0\geqslant \delta/N_0$. 
\end{rem}

\begin{ex}
In the case of torsion-free ``convex cocompact'' groups in the hyperbolic space $\mathbb{H}_n$, the constant $\alpha$ depends only on the injectivity radius. However, if torsion is allowed, then it also depends on the maximal order of a finite subgroup. 
\end{ex}

\begin{rem} 
If $E(U)>10^{14}\kappa_0  \log_2(2|U|) $, then $\lambda_0(U)\geqslant 10^{14}\kappa_0 \log_2(2|U|)$ so that the theorem applies.
 \end{rem}

\begin{rem}\label{IR: acylindrical 2}
The logarithmic term in Theorem \ref{IT: acylindrical hyperbolic} is due to the correction term in Gromov's tree approximation lemma, see Theorem \ref{T: tree approximation}. Due to the work of Alice Kerr \cite{kerr_tree_2020}, if X is a \emph{quasi-tree} in the sense of Bestvina-Bromberg-Fujiwara \cite{bestvina_constructing_2015}, then the tree approximation is uniform, and this logarithmic term is not necessary.  Important examples of groups acting on quasi-trees are discussed in \cite{bestvina_constructing_2015}. In many cases, one can therefore improve Theorem \ref{IT: acylindrical hyperbolic}, and can get rid of the logarithmic term in Theorem \ref{IT: acylindrical hyperbolic}. 
\end{rem}

\subsection{Strategy of proofs} \label{S: strategy of proofs}

Let us fix $x_0$  as a point minimising (up to an error of $\delta$) the normalised $\ell^1$--energy $$E(U)(x):=\frac{1}{|U|}\sum_{u\in U} |x-ux|.$$ 
 
As in the case of free products \cite{button_explicit_2013}, the proof of Theorem \ref{IT: hyperbolic groups}, \ref{IT: acylindrical trees} and \ref{IT: acylindrical hyperbolic} splits into two cases: 
\begin{description}
\item[Case 1 (Concentrated energy)] for more than $1/4$ of the elements of $U\subset G$,  $$|ux_0-x_0|\leqslant 10^{-4}\lambda_0(U),$$
\item[Case 2 (Diffuse energy)]  for at least $3/4$ of the elements of $U\subset G$, $$|ux_0-x_0|>10^{-4}\lambda_0(U).$$
\end{description} 

As in \cite[Lemma 2.4]{button_explicit_2013}, the proof in Case 1 of concentrated energy relies on acylindricity and a standard ping pong argument, see Lemma \ref{L: small displacement} for details. 

Case 2 of diffuse energy is more involved. In this case, we closely follow the strategy of Safin \cite{safin_powers_2011} in the case of a free group. The proof is based on two definitions: reduced products and periodic elements. Hyperbolic geometry enables us to extend it to the case of a group acting on a hyperbolic space.  

Let us sketch this proof in the case where $X$ is a $\delta$--hyperbolic  graph and the action of $G$ on $X$ is cocompact and has a large injectivity radius (i.e. $\inf_{g\not=1,x\in X}|gx-x|>10^{14}\kappa_0$).

The \emph{Gromov product} of $p,\, q \in X$ at $x\in X$ is $(p,q)_x:=\frac12 \left(|p-x| + |q-x| -|p-q|\right)$.
\begin{df}[cf. Definition \ref{D: reduced products}]
The product $gh$ is reduced at $x_0$ if $(g^{-1}x_0,hx_0)_{x_0}\leqslant \delta$.
\end{df}

In the case of a free group acting on its Cayley graph with $x_0=1$, the product of two elements $g$ and $h$ (we represent $g$ and $h$ in normal form) is reduced if the last letter of $g$ is not inverse to the first letter of $h$.

Our first lemma generalises \cite[Lemma 1]{safin_powers_2011}. 

\begin{lem}[cf. Lemma \ref{L: reduction}]\label{IL: reduction} Let $\alpha= 100 \,  |B(1000 \delta)|^2$ and $U\subset G$ a finite set.  
Then there are two subsets $U_1$ and $U_2$ of $U$ of cardinality $>\frac{1}{\alpha} $ such that all products $uv$ and $vu$ where  $u\in U_1$ and $v\in U_2$ are reduced at $x_0$.
\end{lem}

We now discuss our second definition, \emph{periodic isometries}.
The axis $C_u$ of an hyperbolic isometry $u\in G$ is the set of $x\in X$ such that $|x-ux| \leqslant \inf_{x\in X} |x-ux | +8 \delta$.  The axis is a $\delta$--quasi-convex $u$--invariant subset that contains a bi-infinite line at finite Hausdorff distance \cite[Proposition 2.3.5]{delzant_courbure_2008}. 

\begin{df}[cf. Definition \ref{D: periodic}]\label{ID: periodic} An isometry $g\in G$ is \emph{$u$--periodic at $x_0$} if $x_0$ and $gx_0$ lie in $C_u$ and if $|x_0-gx_0|>2 \inf_{x\in X} |x-ux|+1000 \delta$.  The \emph{period} of $g$ is the maximal cyclic subgroup containing $u$. 
\end{df}
The period of $g$ at $x_0$ is unique, that is, it does only depend on $x_0$ and $g$, but not on $u$. 

In the case of a free group acting on its Cayley graph with $x_0=1$, $1\in C_u$ if  $u$ is cyclically reduced, that is, its first letter is not inverse to its last letter, and in this case, $g$ is $u$-periodic at the origin  if $g=u^ku_1$ where $\kappa\geqslant 2$, $u_1$ is non-empty and $u=u_1u_2$ as a (cyclically) reduced product. Therefore, our definition  is a geometric adaption of the notion of a periodic word, a key concept in small cancellation theory, cf. \cite[Section 2]{safin_powers_2011}.

Our second main lemma generalises another  observation of Safin \cite[Lemma 3]{safin_powers_2011}.

\begin{lem}[cf. Lemma \ref{L: periodic}]\label{IL: small cancellation} Let  $u_1vw_1=u_2vw_2=u_3vw_3=u_4vw_4$ where all  the products $u_iv$ and $vw_i$ are reduced at $x_0$ and $|vx_0-x_0|\geqslant |u_ix_0-u_jx_0|$, where $1\leqslant i,j \leqslant 4$. Then $v$ is periodic  at $x_0$. 
More precisely, the period of $v$ is the maximal cyclic subgroup containing $u_1^{-1}u_2$.  
\end{lem}

With these two lemmas in hands, we mimic Safin's final counting argument \cite[Section 4]{safin_powers_2011} to estimate $|U_1U_2U_1|\geqslant |U_1vU_1|$. For instance, if   $v\in U_2$ is not periodic at $x_0$, then $|U_1U_2U_1|\geqslant |U_1vU_1|\geqslant \frac{1}{4} |U_1|^2$ by the previous lemma. Otherwise, all elements in $U_2$ are periodic with same period and same tale, see details in Section \ref{S: Periodicity}. The result then follows by  a ping pong argument, see details in Section \ref{S: Estimation}.
\medskip

The main difficulty, when one wants to extend this result to groups acting on a hyperbolic space is the choice of two large subsets of $U$ such that all products are reduced. In the case of trees this difficulty was resolved by Button \cite{button_explicit_2013}. We adapt Button's argument with help of Gromov's tree approximation lemma  \cite[Section 6.1]{gromov_hyperbolic_1987},  Theorem \ref{T: tree approximation}, to the case of hyperbolic spaces.

\subsection{Structure of the paper} 

The paper has five parts. In Section \ref{S: small cancellation theory}, we discuss reduced products and periodic elements, familiar in small cancellation theory, but we adapt these notions to our geometric setting. In Section \ref{S: Estimation}, we discuss the case of sets of elements of same period and same tale. 
In Section \ref{S: energy discussion}, we discuss the energy and displacement of $U$ and fix the base point $x_0$. The proof now splits into the aforementioned cases of concentrated energy (Section \ref{S: small displacement}) and diffuse energy at $x_0$ (Section \ref{S: the case of diffuse energy}).

\medskip
 
 \noindent
\textbf{Acknowledgements.} We thank Emmanuel Breuillard, Fran\c{c}ois Dahmani and Audrey Vonseel for drawing our attention to Safin's paper and for important discussions on the topic, and the referees for their careful reading and numerous useful comments that helped us to greatly improve the exposition. T.D. thanks the Isaac Newton Institute (Cambridge) for hospitality during the programme ``Non-positive curvature group actions and cohomology''. M.S. was supported by Labex IRMIA at the University of Strasbourg, ERC-grant GroIsRan no.725773 of A. Erschler and   the Austrian Science Fund (FWF) project J 4270-N35.

\section{Small cancellation theory} \label{S: small cancellation theory}

The main notions used in Safin's argument for free groups \cite{safin_powers_2011} and in Button's generalisation to free products \cite{button_explicit_2013} are reduced products and periodic elements, familiar in small cancellation theory. We recall that the product $gh$ of two elements $g$ and $h$ in a free group is reduced if the last letter of $g$ is not inverse to the first letter of $h$; and $g$ is periodic if it can be written as $(uv)^ku$ where the products $uv$ and $vu$ are reduced. 

In this part our goal is to explain geometric interpretations of reduced products and periodic elements that we can use in the general setting for groups acting on hyperbolic spaces.

\subsection{Hyperbolic spaces}\label{S: hyperbolic and isometries}

We collect some facts of the geometry of hyperbolic spaces in the sense of \cite{gromov_hyperbolic_1987}, see \cite{coornaert_geometry_1990,ghys_hyperboliques_1990,delzant_courbure_2008,coulon_geometry_2014,coulon_partial_2016}.

\subsubsection{Hyperbolic spaces} Let  $X$ be a geodesic metric space. The distance of two points $x$ and $y$ in $X$ is denoted by $|x-y|$, and $[x,y]$ is a geodesic segment  between $x$ and $y$. If $A$ is a set, and $x$ a point, we write $d(x,A)=\inf_{a\in A}\{|x-a|\}$  to denote the distance from $x$ to $A$, and let $A^{+a}:=\{ x\in X \mid d(x,A)\leqslant a \}$. 

We recall that the \emph{Gromov product} of $p,\, q \in X$ at $x\in X$ is defined by $$(p,q)_x:=\frac12 \left(|p-x| + |q-x| -|p-q|\right).$$ 

\begin{df}\label{D: hyperbolicity} Let $\delta\geqslant 0$. The geodesic metric space $X$ is \emph{$\delta$--hyperbolic}, if for every choice of four points $p,\, q,\, r,$ and $x$ in $X$ the inequality $(p,r)_x\geqslant \min \{ (p,q)_x,(q,r)_x\} -\delta$ is satisfied. 
\end{df}
From now on, we assume that $\delta\geqslant 0$ and we assume that $X$ is $\delta$--hyperbolic.
If $\delta=0$, then $X$ is an $\mathbb{R}$--tree \cite[Chapitre 3, Th\'eor\`eme 4.1]{coornaert_geometry_1990}.

We frequently use the following observations.

\begin{rem}\label{R: inversed triangle inequality} If  $x,y,z\in X$, $d\geqslant 0$ and $(x,z)_y\leqslant d\delta $, the following type of inversed triangle inequality holds: $|x-z|\geqslant |x-y|+|y-z| -2 (x,z)_y\geqslant |x-y|+|y-z| - 2d\delta.$
\end{rem}

\begin{lem}\label{L: Gromov product}
If $x,y,z\in X$ then $(x,z)_y\leqslant d(y,[x,z])$. 
\end{lem}
 The proof of this lemma does not rely on hyperbolicity.
\begin{proof} Let $p$ denote a projection of $y$ on $[x,z]$, so that $d(y,[x,z])=|p-y|$. 
\mbox{} Then $(x,z)_y= \frac{1}{2} \left( |x-y|+ |z-y| -|x-z|\right) \leqslant \frac{1}{2} \left( |x-p|+ |z-p| -|x-z|\right) + |y-p| =  (x,z)_p + d(y,[x,z])$. 
But $p$ lies on $[x,z]$, so that $(x,z)_p=0$. We conclude the claim.
\end{proof}

On the other hand, hyperbolicity yields the following.
\begin{lem}[cf. {\cite[Chapitre 1 Proposition 3.1]{coornaert_geometry_1990}}]\label{L: Gromov product 2}
If $x,y,z\in X$ then $d(z,[x,y])\leqslant (x,y)_z + 4\delta$. 
\end{lem}
\begin{proof}
Let $p_1$ be  on a geodesic $[x,z]$ such that $|p_1-z|=(x,y)_z$, and $p_2$ on $[x,y]$ such that $|p_1-x|=|p_2-x|$. 
Then $|p_1-p_2|\leqslant 4\delta$. Indeed, by hyperbolicity, 
$(p_1,p_2)_x\geqslant \min\{(p_1,z)_x,(z,y)_x,(y,p_2)_x \}-2\delta.$ 
As $(z,y)_x=|x-z|-(x,y)_z=|p_1-x|$ and $(p_1,z)_x=|p_1-x|=|p_2-x|=(y,p_2)_x$, 
$(p_1,p_2)_x= |p_1-x|-\frac{1}{2} |p_1-p_2|\geqslant |p_1-x| -2\delta,$
 hence the claim. 
Thus, $d(z,[x,y])\leqslant |p_2-z|\leqslant  (x,y)_z + 4\delta$. 
\end{proof}

\begin{lem}\label{L: thin triangles}
If $x$, $y$, $z$  are points in $X$ and if $p$ is on a geodesic $[x,y]$ such that $(y,z)_x\geqslant |p-x|$, then $(x,z)_p \leqslant \delta$. In particular, $d(p,[x,z])\leqslant 5\delta$.
\end{lem}
\begin{proof}
By definition of hyperbolicity, $(p,z)_x\geqslant \min ((p,y)_x,(y,z)_x) - \delta$.
By assumptions, $(p,y)_x=|p-x|\leqslant(y,z)_x$, so that $(p,z)_x\geqslant |p-x| - \delta$. By definition of the Gromov product, then  
$(x,z)_p = |x-p| -(p,z)_x \leqslant \delta.$ By Lemma \ref{L: Gromov product 2}, $d(p,[x,z])\leqslant 5\delta$.
\end{proof}

\begin{lem}\label{L: thin triangles 2}
If $x$, $y$ and $z$ are points in $X$ and $p$ a projection of $z$ on a geodesic segment $[x,y]$, then 
$(z,x)_p\leqslant 4\delta$, $(y,z)_p\leqslant 4\delta$ and $|x-p|\leqslant (z,y)_x+4\delta$.
\end{lem}
 \begin{proof}
As $p$ is on $[x,y]$, $(y,x)_z\leqslant (p,x)_z$. Thus, $(z,x)_p=|z-p|-(p,x)_z\leqslant |z-p| -(y,x)_z$. By Lemma \ref{L: Gromov product 2},  $(z,x)_p\leqslant 4\delta$. Similarly, $(y,z)_p\leqslant 4\delta$. Finally, by Remark \ref{R: inversed triangle inequality}, $|z-x|\geqslant |z-p|+|p-x| -8\delta$, so that $(z,y)_x\geqslant (p,y)_x-4\delta = |x-p| - 4\delta$, which completes the proof. 
 \end{proof}
 In general, we have the following important theorem of Gromov, called tree approximation lemma.

\begin{thm}[Tree approximation lemma {\cite[Section 6.1]{gromov_hyperbolic_1987} \cite[Chapitre 8, Th\'eor\`eme 1]{coornaert_geometry_1990}}] \label{T: tree approximation}
 Let $x_0$, $x_1$, $\ldots$, $x_n\in X$ and let $A$ be the union of geodesic segments $\bigcup_{i=1}^n[x_0,x_i]$. There is a metric tree $T$ and a map $f:A\to T$ such that 
 \begin{enumerate}
 \item for all $1\leqslant i\leqslant n$, the restriction of $f$ to the geodesic segment $[x_0,x_i]$ is an isometry  and
 \item  for all $x,x'\in A$, $|x-x'|-2\delta (\log_2 (n)+1) \leqslant |f(x)-f(x')|\leqslant  |x-x'|$. 
 \end{enumerate} \qed
\end{thm}

\subsubsection{Stability of local quasi-geodesics}

A path $\gamma: [a,b] \to X$  is a \emph{$1$--quasi-geodesic} if for all $[a',b']\subseteq [a,b]$ the
 $\length \left( \gamma([a',b'])\right) \leqslant  |\gamma(a')-\gamma(b')| + \delta.$
 We call $\gamma$ a \emph{$L$-local $1$--quasi-geodesic} if $L\geqslant 0$ and if the restriction of $\gamma$ to every interval of diameter at most $L$ is a $1$--quasi--geodesic.
 
We use the next lemma to study (Lemmas \ref{L: stability finite}, \ref{L: stability infinite})  and to construct (bi-infinite) local quasi-geodesics (Lemma \ref{L: invariant lines}). Part (1) is also used in ping-pong arguments, for instance to construct free subgroups in hyperbolic groups \cite{delzant_sous_1991,delzant_sous-groupes_1996}. Such arguments will be used in Sections \ref{S: Estimation} and \ref{S: small displacement}.

\begin{lem}[Discrete quasi-geodesics {\cite[7.2.C]{gromov_hyperbolic_1987} \cite{delzant_sous_1991} \cite[Proposition 1.3.4]{delzant_sous-groupes_1996}}] \label{L: discrete quasi-geodesics} Let $(x_i)_i$   be a sequence of points $x_i\in X$.
\begin{enumerate}
\item If for all $i$ 
$$ (x_{i-1},x_{i+1})_{x_i} \leqslant 1/2 \; \min\{ |x_i-x_{i-1}|, \, |x_i-x_{i+1}|\} -\alpha -\delta,$$
then $|x_i-x_j|\geqslant \alpha |i-j|$.
\item If, in addition, $\alpha \geqslant 9\delta$ and if for all $i$ we have $(x_{i-1}, x_{i+1})_{x_i}\leqslant \beta$, then for every $m\geqslant n+1$ the piece-wise geodesic 
$$ [x_n,x_{n+1}]\cup [x_{n+1},x_{n+2}] \cup \ldots \cup [x_{m-1},x_m] $$
is at distance at most $10\delta +\beta$ from any geodesic $[x_n,x_m]$. Inversely, every point on $[x_n,x_m]$  is at distance at most $10\delta + \beta$ from this piece-wise geodesic. \qed
\end{enumerate} 
\end{lem}

Lemma \ref{L: discrete quasi-geodesics}(2) yields: 
 
\begin{lem}[Stability of local quasi-geodesics, cf. {\cite[Th\'eor\`eme 2.1.4]{delzant_courbure_2008}}] \label{L: stability finite} 
The Hausdorff distance of any two $200\delta$-local $1$--quasi--geodesics with same endpoints in $X$ is at most $22\delta$.  
\end{lem}
\begin{proof}
Let $\gamma$ be a $200\delta$-local $1$--quasi--geodesic from $x$ to $y$. Take $x_0:=x, \ldots ,x_i, \ldots, x_n:=y$ on $\gamma$ such that $50\delta \leqslant |x_i-x_{i+1}|\leqslant 100\delta$. As $x_{i-1}$, $x_i$, and $x_{i+1}$ are on a $1$--quasi--geodesic we have that $(x_{i-1},x_{i+1})_{x_i}\leqslant \delta$. By Lemma \ref{L: discrete quasi-geodesics}(2) with $\beta = \delta$ and $\alpha = 9 \delta$, the Hausdorff distance between $\gamma$ and $[x,y]$ is at most $11\delta$. Therefore, two $200\delta$-local $1$--quasi--geodesics from $x$ to $y$ are of Hausdorff distance at most $22\delta$.
\end{proof}

We denote by $\partial  X$ the \emph{boundary at infinity of $X$}, see \cite[Chapitre 2]{coornaert_geometry_1990}. The stability of local quasi-geodesics extends to bi-infinite local quasi-geodesics with endpoints in $\partial  X$.

\begin{lem}[cf. {\cite[Chapitre 3, Th\'eor\`eme  3.1]{coornaert_geometry_1990}, \cite[Corollary 2.6]{coulon_geometry_2014}}] \label{L: stability infinite}
The Hausdorff distance of any two $200\delta$-local $1$--quasi--geodesics with same endpoints in $X\cup \partial  X$ is at most $60\delta$.  \qed
\end{lem}

\subsubsection{Isometries} Let $g$ be an isometry of $X$. 

\begin{df}[Translation length] The \emph{translation length $[g]$ of $g$} is 
$$[g]:=\inf_{x\in X} |gx-x|.$$
\end{df}

If $[g]>16\delta$, then $g$ is a hyperbolic isometry \cite[Chapitre 10,  Propositions 6.3 et 6.4]{coornaert_geometry_1990}.

\begin{df}[\hspace{-0.15mm}{\cite[D\'efinition 2.3.2]{delzant_courbure_2008} \cite[Definition 2.27]{coulon_geometry_2014}}] Let 
$$C_g:=\{ x\in X\mid |gx-x|\leqslant [g]+8 \delta\}.$$ 
\end{df}

A subset $C\subset X$ is \emph{$\alpha$--quasi-convex} if for all two points $x,y\in C$ the distance of each point on a geodesic segment $[x,y]$ to $C$ is at most $\alpha$. For instance, geodesic segments or lines are $2\delta$--quasi-convex. 

\begin{lem}[{\hspace{-0.15mm}\cite[Proposition 2.3.3]{delzant_courbure_2008} \cite[Proposition 2.28]{coulon_geometry_2014}}] \label{L: axis} Let $g$ be an isometry of $X$.  Then $C_g$ is $10\delta$--quasi-convex and $g$--invariant. Moreover, for all $x\in X$, 
$$[g]+ 2d(x, C_g) - 10\delta \leqslant |gx-x| \leqslant [g] + 2d(x, C_g) +10\delta.$$ \qed
\end{lem}

In the case of a tree or a negatively curved (CAT(-1)) space and a hyperbolic isometry $g$, the set $C_g$ is a neighbourhood of the bi-infinite geodesic $g$--invariant line $L_g$, the axis of $g$.  In general, $L_g$ is a $[g]$--local $1$--quasi-geodesic, and we have:

\begin{lem}[cf. {\cite[Proposition 2.3.5]{delzant_courbure_2008}}]\label{L: invariant lines}
Let $g$ be an isometry with $[g]>200 \delta$ and let $x\in C_g$ with $|gx-x|\leqslant [g] + \delta$. Then the piece-wise geodesic $L_g:= \bigcup _{n\in \mathbb{Z}} [g^nx,g^{n+1}x]$ is a $[g]$--local $1$--quasi--geodesic bi-infinite $g$--invariant line. 
The set of endpoints of $L_g$ in $\partial X$ is fixed by $g$. Moreover, $L_g$ is $11\delta$--quasi-convex. In particular, $C_g\subset L_g^{+30\delta}$. 
\end{lem}

\begin{proof} Let us first note that $[g]\leqslant |g^nx-g^{n+1}x|\leqslant [g] +\delta$. Let $m_n$ be the midpoint of $[g^nx,g^{n+1}x]$. Then $|m_n-m_{n+1}|\leqslant [g] + \delta$ and as $m_{n+1}=gm_n$ also $|m_n-m_{n+1}|\geqslant [g]$. Hence, $L_g$ is a $[g]$--local $1$--quasi--geodesic. Using Lemma \ref{L: discrete quasi-geodesics} as in the proof of Lemma \ref{L: stability finite}, we see that $L_g$ is indeed a bi-infinite line. It is $g$--invariant by definition. We argue similarly to see that $L_g$ is quasi-convex: for every two points $z$, $y$ of distance at least $200\delta$ on $L_g$, the proof of Lemma \ref{L: stability finite} implies that $[z,y]$ is at Hausdorff distance at most $11\delta$ from $L_g$. 
Now, if $x\in C_g$, then $d(x,L_g)\leqslant 30\delta$. Indeed, otherwise quasi-convexity and $g$--invariance of $L_g$ would imply that $|x-gx| > [g]+8\delta$, which is a contradiction. Hence, $C_g\subset L_g^{+30\delta}$.
\end{proof}
\begin{rem} \label{R: neighbourhood quasi-convex}
If $\alpha \geqslant 10\delta$, then $C_g^{+\alpha}$ is $10\delta$--quasi-convex, see {\cite[Chapitre 10, Proposition 1.2]{coornaert_geometry_1990}}. Similarly, if $\alpha \geqslant 11\delta$, then $L_g^{+\alpha}$ is $11\delta$--quasi-convex.
\end{rem}
Furthermore, we have:

\begin{lem}\label{L: midpoint} Let $g$ be an isometry of $X$ and let $x,y \in X$.  
If $m_1 $ is the midpoint of $[x,gx]$ and $m_2$ is the midpoint of $[y,gy]$, then $m_1 \in C_g^{+60\delta}$ and $m_2\in C_g^{+60 \delta}$. Thus  $[m_1,m_2] \subset C_g^{+ 70 \delta}.$
\end{lem}

\begin{ex} Let us illustrate this for the free group acting on its Cayley graph. Then the vertices $1$ and $g$ are in $C_g$ if and only if $g$ is cyclically reduced, that is, the first  and last letter of $g$ are not inverse to each other. If $g$ is not cyclically reduced, then $g=hg'h^{-1}$ for some non-trivial element $h$ and a cyclically reduced $g'$. As $1$ and $g'$ and therefore the midpoint of the geodesic between these points  lie in $C_{g'}$, the midpoint of the geodesic from $1$ to $g$ lies in $C_g$. 
\end{ex}

\begin{proof} 
If  $d(x, C_g)\leqslant 50\delta$, then $d(gx,C_g)\leqslant 50\delta$ as $C_g$ is $g$-invariant (Lemma \ref{L: axis}). Thus $m_1$ is in $C_g^{+60\delta}$ by $10\delta$--quasi-convexity of $C_g^{+50\delta}$ (Remark \ref{R: neighbourhood quasi-convex}).

Otherwise, let $p \in C_g^{+23\delta}$ such that $d(x,C_g^{+23\delta}) - \delta/2 \leqslant |x-p|\leqslant d(x,C_g^{+23\delta})$.  
 By Lemma \ref{L: axis}, $gp\in C_g^{+23\delta}$ and  $|p-gp|\geqslant 34 \delta$. 
As $C_g^{+23\delta}$ is $10\delta$--quasi-convex, then (by \cite[Chapitre 10, Proposition 2.1]{coornaert_geometry_1990}), 
$|x-gx|\geqslant 2|x-p|+|p-gp|-33\delta \geqslant 2|x-p|$. Thus,  $|p-x|\leqslant |m_1-x|$.

Then $(p,gx)_x\leqslant |p-x|\leqslant |m_1-x|$, and as $(p,x)_{gx}=|gx-x|-(p,gx)_x$, $(p,x)_{gx}\geqslant |m_1-gx|$. Thus, $d(m_1,[p,gx])\leqslant 5\delta$ by Lemma \ref{L: thin triangles}.

If $m_1'$ denotes  a projection of $m_1$ on $[p,gx]$, then  $(p,gp)_{gx}\leqslant |gx-gp|\leqslant |m_1-gx| \leqslant  |m_1'-gx|+5\delta$.  
 Therefore, $(gx,gp)_p=|p-gx| -(p,gp)_{gx}\geqslant |m_1'-p|-5\delta$. Thus, by Lemma \ref{L: thin triangles}, $d(m_1',[p,gp])\leqslant 5\delta +5\delta =10\delta$. We conclude that $d(m_1,[p,gp])\leqslant 15 \delta$. 

By $10\delta$--quasi-convexity (Remark \ref{R: neighbourhood quasi-convex}) of $C_g^{+23\delta}$, $m_1\in C_g^{+48\delta}\subset C_g^{+60\delta}$.  Analogously $m_2\in C_g^{+60\delta}$. By $10 \delta$--quasi-convexity (Remark \ref{R: neighbourhood quasi-convex}), $[m_1,m_2]\subset C_g^{+ 70 \delta}$. This completes the proof. 
\end{proof}

\subsection{Acylindrical actions on hyperbolic spaces}\label{S: acylindrical actions}

In a free group acting on its Cayley graph, or more generally, a group acting $\kappa$--acylindrically on a tree, two hyperbolic elements $h$ and $g$ commute if and only if the diameter of the intersection of their axis is larger than $[g]+[h]+\kappa$. Indeed, the commutator $[g,h]$ then pointwise stabilises  a geodesic segment of length $\kappa$ in the intersection of the axis, and by acylindricity this is possible only if $g$ and $h$ commute.

In this section we discuss a well known generalisation  of this fact for groups acting on hyperbolic spaces (Lemma \ref{L: small cancellation}), see  \cite{delzant_courbure_2008,coulon_geometry_2014,coulon_partial_2016}. We later use this to show that periods are unique.

\subsubsection{Acylindrical actions}  

We use the following definition of acylindricity of {\cite[Proposition 5.31]{dahmani_hyperbolically_2017}. 
\begin{df}[Acylindrical action]\label{D: acylindrical} Let $\kappa_0 \geqslant \delta  \hbox{ and } N_0 \geqslant 1.$ The action of $G$ on the $\delta$--hyperbolic space $X$ is \emph{$(\kappa_0, N_0)$--acylindrical} 
 if for all $x,y\in X$ of distance $|x-y|\geqslant \kappa_0$ there are at most $N_0$ isometries $g\in G$  with 
$|gx-x|\leqslant 100\delta \hbox{ and } |gy-y|\leqslant 100 \delta.$
\end{df}

\begin{lem}[{\hspace{-0.15mm}\cite[Proposition 5.31]{dahmani_hyperbolically_2017}}] \label{L: acylindrical} Let $d\geqslant 1$. If  
$$ \kappa(d):= \kappa_0 + 400 d \delta +100 \delta  \hbox{ and } N(d):= 23 d N_0,$$
then for all $x,y\in X$ of distance $|x-y|\geqslant \kappa(d)$ there are at most $N(d)$ isometries $g\in G$  with 
$$|gx-x|\leqslant 100d\delta \hbox{ and } |gy-y|\leqslant 100 d\delta.$$ \qed
\end{lem}

\begin{rem}\label{R: acylindricity trees} Let $G$ act  $\kappa$--acylindrically \cite{sela_acylindricity_1997} on a tree $X$ of edge-length $\rho_0$. 
 Then $X$ is $0$--hyperbolic, and the action of $G$ on $X$ is $(\kappa,1)$--acylindrical.
\end{rem}

\begin{ex}Let $G$ split as a free product $G=A*_C B$ amalgamated over a subgroup $C$. Then $G$ acts on the Basse-Serre tree of this splitting. If $C$ is finite, the edge stabilisers are conjugate to $C$, therefore for all $\varepsilon > 0$ the action is $(\varepsilon ,|C|)$--acylindrical. If $C$ is malnormal in $A$, the action is $(2 ,1)$--acylindrical.
\end{ex} 

\begin{ex} Let $G$ be a hyperbolic group that acts properly and cocompactly on a $\delta$--hyperbolic space $X$. Let $b_0$ be a uniform bound on the number of $g\in G$ such that for all $x\in X$, $|gx-x|\leqslant 102 \delta$. Then $G$ acts $(\delta, b_0)$--acylindrically on $X$. If $X$ is a Cayley graph of $G$, then $b_0=|B(1,102\delta)|$. 
\end{ex}

\begin{df} \label{D: injectivity radius}
If $\delta>0$ and $X$ is $\delta$--hyperbolic we let $\rho_0:= \delta/N_0$ and if $X$ is a simplicial tree, we let $\rho_0$ be the edge length of $X$. 
\end{df}
The \emph{stable injectivity radius} of the action of $G$, denoted by  $\rho$, is the infimum of the stable translation lengths $lim_{n\to \infty} \frac{1}{n} |g^nx-x|$ of hyperbolic elements $g\in G$.  If $X$ is a simplicial tree, then $\rho$ is always at least the edge-length. In general, we have:
\begin{rem}[{\hspace{-0.15mm}\cite[Lemma 2.2]{bowditch_tight_2008}, cf.  \cite[Lemma 3.9]{coulon2017small}}]\label{R: injectivity radius} 
If the action of $G$ on $X$ is $(\kappa_0,N_0)$--acylindrical, then $\rho \geqslant \rho_0$. 
\end{rem}

From now on we assume that $\delta >0$ and that $\kappa_0\geqslant \delta$, or that $X$ is a simplicial tree of edge length $\rho_0$ and $\kappa_0\geqslant \rho_0$. Moreover, we assume that the action is $(\kappa_0,N_0)$--acylindrical in the sense of  Definition \ref{D: acylindrical}.

\begin{rem} If $X$ is a compact Riemannian manifold of constant negative curvature whose injectivity radius is smaller than the Margulis constant,  then $\kappa_0$ can be taken to be the diameter of the thin part.  In our setting there is no lower bound on the curvature, but the acylindricity assumption replaces the Margulis lemma.
\end{rem}

\subsubsection{Loxodromic subgroups}

The limit set of $G$ acting on $X$ is the set of accumulation points in $\partial X$ of one, and hence all, orbits of $G$. A subgroup of $G$ is \emph{elementary} if its limit set consists of at most two points. Elementary subgroups are  discussed, for instance, in \cite[Section 3.4]{coulon_partial_2016}.  

A \emph{loxodromic} subgroup $E$ of $G$ is an elementary subgroup containing a hyperbolic isometry. It is contained in a unique maximal loxodromic subgroup. The action of $E$ on $\partial X$ fixes a subset of exactly two points, $x^+$ and $x^-$, which are those in the limit set of $E$. The maximal loxodromic subgroup containing $E$ is the stabiliser of the set $\{ x^+, x^-\}$. 
We define  $$[E]:=\min \{ [g] \mid g\in E \hbox{ with } [g]> 200\delta\}.$$ 
 For a hyperbolic $h\in G$, we denote by $E(h)$ the \emph{maximal loxodromic subgroup} containing $h$.

 If $G$ is a hyperbolic group, or more generally, if $\delta>0$ and $G$ acts $(\kappa_0,N_0)$--acylindrically on $X$, then maximal loxodromic subgroups are maximal infinite virtually cyclic subgroups \cite[Proposition 3.27]{coulon_partial_2016}.

\begin{df}[Invariant cylinder] Let $E$ be a loxodromic subgroup with limit set $\{ x^+,x^-\}$.  The \emph{$E$--invariant cylinder}, denoted by $C_E$, is the $40\delta$--neighbourhood of all $200 \delta$-local  $1$--quasi--geodesics with endpoints $x^+$ and $x^-$ at infinity. 
\end{df}

\begin{ex}
If $G$ is acting on a tree, $C_E$ coincides with the axis of any hyperbolic element in $E$, which is a bi-infinite geodesic line, with endpoints $x^+$ and $x^-$ at infinity. 
\end{ex}

In general, we think of $C_E$ as of a quasi-convex thickening of the axis of a hyperbolic element $g$ of $E$. As such it does only depend on the maximal loxodromic subgroup containing $E$. 

\begin{ex} Let $X$ be the hyperbolic plane of  curvature $-1$ and $g$ a hyperbolic isometry of translation length $\varepsilon$. Then $C_g$ is the $1/\varepsilon$--neighbourhood of the $g$--invariant geodesic $L$ whereas $C_E$ is the $1$--neighbourhood of this geodesic.
\end{ex}

If $[g]>200\delta$, we denote by  $L_g$  a $[g]$--local geodesic bi-infinite $g$--invariant line constructed in Lemma \ref{L: invariant lines}. The next lemma shows that $L_g$ and $C_E$ are close as sets in the Gromov-Hausdorff distance.

\begin{lem}[Invariant cylinder, cf. {\cite[Lemma 3.12]{coulon_partial_2016}}]\label{L: invariant cylinder}
Let $E$ be a loxodromic subgroup. Then
\begin{itemize}
\item $C_E$ is $10\delta$--quasi-convex and invariant under the action of $E$, 
\item if $g\in E$ and $[g]>200\delta$, then $L_g\subset C_g\subset C_E$ and $C_E\subset L_g^{+100\delta}\subset C_g^{+ 100\delta}$. In particular, if $x_0\in C_E$,  then $|gx_0-x_0|\leqslant [g] + 210 \delta $.
\end{itemize}     
\end{lem}

\begin{proof} The cylinder $C_E$ is invariant by definition. It is quasi-convex as  $200 \delta$-local  $1$--quasi--geodesic bi-infinite lines are quasi-convex and any two such lines with same endpoints are at finite Hausdorff distance, see \cite[Lemma 3.12]{coulon_partial_2016}. Let $g\in E$ and $[g]>200\delta$. By Lemma \ref{L: invariant lines}, $L_g\subset C_g\subset L_g^{+30\delta}$. With Lemma \ref{L: stability infinite}, this implies the second assertion. 
\end{proof}

\subsubsection{Small cancellation lemma}
We recall that $G$ acts $(\kappa_0,N_0)$--acylindrically on the $\delta$--hyperbolic space $X$ and $\rho_0$ is the lower bound on the stable injectivity radius of the action, see Definition \ref{D: injectivity radius} and Remark \ref{R: injectivity radius}.
 We define 
 $$\nu:=4N_0\frac{\kappa_0}{\rho_0}\hbox{ and } A:=10^7  N_0^2 \frac{\kappa_0}{\rho_0} .$$

We let $E$ be a maximal loxodromic subgroup, and recall that $$[E]=\min \{ [g] \mid g\in E \hbox{ with } [g]> 200\delta\}.$$
\begin{lem} \label{L: small cancellation}
If $E\not= E'$, then 
$$\diam(C_E^{+400 \delta}\cap C_{E'}^{+ 400 \delta}) \leqslant   3\nu \, \max\left\{[E],[E']\right\}  + A \delta + 1684 \delta .$$
\end{lem}

\begin{proof} Let $g\in E$ realise $[E]$, and let $g'\in E'$ realise $[E']$. Let us recall that $C_E\subset C_g^{+100 \delta}$ and that $C_{E'}\subset C_{g'}^{+100 \delta}$, see Lemma \ref{L: invariant cylinder}. By
 \cite[Proposition 3.44]{coulon_partial_2016} (together with \cite[Lemma 2.16]{coulon_partial_2016})
 $$\diam\left( C_g^{+ 500 \delta}\cap C_{g'}^{ +500 \delta}\right) \leqslant [g]+[g'] + \nu(G,X) \max\left\{[g],[g']\right\} + A(G,X) + 1684 \delta.$$
By \cite[Section 6.2]{coulon_partial_2016},  $ \nu(G,X) \leqslant N_0+\kappa_0/\rho+2\leqslant \nu$. Moreover, if $\delta >0$ (cf. \cite[Lemma 3.12]{coulon2017small}, in their notation, $\kappa_0=L$, $N_0=N$ and we take  $L_S:=200$) we have:
$ A(G,X) \leqslant 10\cdot 200^2 N_0^3(\kappa_0+ 5\delta) \leqslant A\delta.$ If $X$ is a simplicial tree $A(G,X)=0$ by its definition \cite[Definition 3.40]{coulon_partial_2016}, which yields the claim. 
\end{proof}

\subsection{Reduced products and periodic isometries} \label{S: reduced and periodich}

We  discuss reduced products and periodic elements. These concepts are familiar in combinatorial small cancellation theory, for example in Novikov-Adian's famous study of the Burnside problem. In this paragraph we develop a geometric point of view.

\subsubsection{Reduced products} 

\begin{df}[Reduced products]\label{D: reduced products} Let $u$ and $v$ be two isometries of $X$. Their product $uv$  is \emph{reduced at the point $x_0 \in X$} if 
$$(u^{-1}x_0,vx_0)_{x_0}\leqslant \delta .$$
\end{df}

In the case of a free group, this means that the first letter of $v$ is not inverse to the last letter of $u$. In a hyperbolic space this means that the point $ux_0$ is $5\delta$ close to a geodesic $[x_0,uvx_0]$, see Lemma \ref{L: Gromov product 2}.

Our main observation concerning reduced products is the following Proposition, which extends  \cite[Lemma 3]{safin_powers_2011}, and for free products \cite[Lemma 2.5]{button_explicit_2013}.  
\begin{prop} \label{P: reduced products} Let $u_1vw_1=u_2vw_2$ and assume that the products $u_1v$, $u_2v$, $vw_1$ and $vw_2$ are reduced at ${x_0}$. If 
$$\hbox{ $|v{x_0}-{x_0}|>26 \delta$,   $ |u_1{x_0}- u_2{x_0}|\leqslant |v{x_0}-{x_0}|$, and   $|u_1{x_0}-{x_0}|\leqslant |u_2{x_0}-{x_0}|$}$$ 
 then 
$$\hbox{${x_0}$ and $v{x_0}$ are in } C_{u_1^{-1}u_{2}}^{+190\delta}.$$

Moreover, $(x_0,u_2x_0)_{u_1x_0}\leqslant 24 \delta$, $(u_1x_0,u_1vx_0)_{u_2x_0}\leqslant 66 \delta$ and $(u_2x_0,u_2vx_0)_{u_1vx_0}\leqslant 138 \delta$.
\end{prop}

We think of  $ C_{u_1^{-1}u_{2}}^{+190 \delta}$ as of a thickening of an axis of $u_1^{-1}u_2$, see Lemma \ref{L: invariant lines}. 
This follows from the quasi-convexity of  $C_{u_1^{-1}u_{2}}$, see Lemma \ref{L: axis}. 

Let us first explain the case where $X$ is a tree ($\delta=0$). Let $g=u_1vw_1=u_2vw_2$ as in the proposition. Then $u_1x_0$, $u_2x_0$, $u_1vx_0$ and $u_2vx_0$ lie in this order on the geodesic $[x_0,gx_0]$. Moreover, the midpoints $m_1$ of $[u_1x_0,u_2x_0]$ and $m_2$ of $[u_1vx_0,u_2vx_0]$ lie in $C_{u_2u_1^{-1}}$ by Lemma \ref{L: midpoint}. This axis is a geodesic line, and the geodesic segment $[m_1,m_2]$ is in this line. We conclude that $u_2x_0$ and $u_1vx_0$ are in the axis, hence, so are $u_1x_0$ and $u_2vx_0$. As $u_1^{-1}C_{u_2u_1^{-1}}= C_{u_1^{-1}u_{2}}$ this concludes the proof in this case. 
 
 In general hyperbolic spaces the proof is similar.

\begin{lem}\label{L: reduced products and geodesics} Under the assumptions of Proposition \ref{P: reduced products}, the points $u_1x_0$, $u_2x_0$, $ u_1vx_0$ and $u_2vx_0$ are $6\delta$ close to a geodesic $[x_0,gx_0]$. 
  \end{lem}
\begin{proof} As the product $u_1v$ is reduced, $(x_0,u_1vx_0)_{u_1x_0}\leqslant \delta$, and as $vw_1$ is reduced, $ (u_1x_0,gx_0)_{u_1vx_0}\leqslant \delta$. 

By definition of hyperbolicity, 
$$ \min \{ (x_0,gx_0)_{u_1x_0},(gx_0,u_1vx_0)_{u_1x_0}\}\leqslant (x_0,u_1vx_0)_{u_1x_0} +\delta  \leqslant 2\delta.$$ 
As $|vx_0-x_0|>3 \delta$,  
$$(gx_0,u_1vx_0)_{u_1x_0}=|u_1vx_0-u_1x_0|-(u_1x_0,gx_0)_{u_1vx_0}>2\delta. $$
Therefore $ (x_0,gx_0)_{u_1x_0} \leqslant 2\delta$ and $u_1x_0$ is $6\delta$ close to $[x_0,gx_0]$, see Lemma \ref{L: Gromov product 2}. 
 
 Similarly, the point $u_1vx_0$ is $6\delta$ close to $[x_0,gx_0]$.  The same arguments apply to $u_2x_0$ and $u_2vx_0$, which yields the assertion.
\end{proof}

\begin{proof}[Proof of Proposition \ref{P: reduced products}] Let $\gamma=[x_0,gx_0]$ be a geodesic. 
 We first construct four points $p_1$, $p_2$, $q_1$, and $q_2$  that lie in this order on $\gamma$ and such that $|u_1x_0-p_1|\leqslant 6 \delta$, $|u_2x_0-p_2|\leqslant 18\delta$, $|u_1vx_0-q_1|\leqslant 42 \delta$ and $|u_2vx_0-q_2|\leqslant 78 \delta$.

Let $p_1$ be a point on $\gamma$ that is closest to $u_1x_0$. Then $|u_1x_0-p_1|\leqslant 6\delta $ by Lemma \ref{L: reduced products and geodesics}. 

Let $p_2'$ be  a point on $\gamma$ that is closest $u_2x_0$. If $p_2'$ is on the right of $p_1$ we let $p_2=p_2'$. Otherwise, we let $p_2=p_1$ and as $|u_2x_0-x_0|\geqslant |u_1x_0-x_0| \geqslant |p_1-x_0| -6\delta$,  $|p_2'-p_1|\leqslant 12\delta$. Hence, $|u_2x_0-p_2|\leqslant 18\delta$. 

We have fixed $p_1$ and $p_2$ as required. Next we fix $q_1$.

Let $q_1'$ be  a point on $\gamma$ that is closest $u_1vx_0$.  As $|vx_0-x_0|>14\delta$, $q_1'$ is on the right of $p_1$. Indeed, otherwise, by the triangle inequality, $2(u_1vx_0,x_0)_{u_1x_0}\geqslant |vx_0-x_0| -|u_1x_0-p_1|-|u_1vx_0-q_1'|+|q_1'-p_1|>2\delta$, a contradiction. If $q_1'$ is on the right of $p_2$ we let $q_1=q_1'$. Otherwise, we let $q_1=p_2$ and  
\begin{align*}
|q_1'-p_2|&  = |p_1-p_2| - |p_1-q_1'| \\
& \leqslant |u_1x_0-u_2x_0| - |vx_0-x_0| +36\delta \leqslant 36 \delta, 
\end{align*}
so that $|u_1vx_0-q_1|\leqslant 42 \delta $. 

Finally, we fix $q_2$: let $q_2'$  be  a point on $\gamma$ that is closest $u_2vx_0$.  As $|vx_0-x_0|>26\delta$, $q_2'$ is on the right of $p_2$. Indeed, otherwise, $2(u_2vx_0,x_0)_{u_2x_0}\geqslant |vx_0-x_0| -|u_2x_0-p_2|-|u_2vx_0-q_2'|+|q_2'-p_2|>2\delta$, a contradiction.   If $q_2'$ is on the right of $q_1$ we let $q_2=q_2'$. Otherwise, we let $q_2=q_1$ and 
\begin{align*}
|q_2'-q_2|&=|q_1-p_1|-|q_2'-p_2|-|p_1-p_2|\\
& \leqslant |u_1vx_0-u_1x_0|-|u_2vx_0-u_2x_0| +72\delta \leqslant 72 \delta,
\end{align*}
so that $|u_2vx_0-q_2|\leqslant 78 \delta$. 

We have fixed $p_1$, $p_2$, $q_1$ and $q_2$ on $\gamma$ as required. Using the triangle inequality, we conclude that $(x_0,u_2x_0)_{u_1x_0}\leqslant (x_0,p_2)_{p_1} +|p_1-u_1x_0|+|u_2x_0- p_2| \leqslant 24 \delta,$ and, similarly, that $(u_1x_0,u_1vx_0)_{u_2x_0}\leqslant 66 \delta$ and $(u_2x_0,u_2vx_0)_{u_1vx_0}\leqslant 138 \delta$.
 
Now let $m_1$ be the midpoint in $[u_1x_0,u_2x_0]$ and let $m_1'$ be the midpoint in  $[p_1,p_2]$. Then (see \cite[Chapitre 10, Corollaire 5.2]{coornaert_geometry_1990}) $|m_1-m_1'|\leqslant \frac{1}{2}(|u_1x_0-p_1|+|u_2x_0-p_2|)+8\delta\leqslant 20\delta.$ 
Similarly, let $m_2$ be the midpoint in $[u_1vx_0,u_2vx_0]$, let $m_2'$ be the midpoint in  $[q_1,q_2]$, and $|m_2-m_2'|\leqslant  \frac{1}{2}(|u_1vx_0-q_1|+|u_2vx_0-q_2|)+8\delta\leqslant 68\delta$. 

By Lemma \ref{L: midpoint} $m_1$ and $m_2$ are in $C_{u_2u_1^{-1}}^{+70\delta}$, so that  $m_1'$ and $m_2'$ are in  $C_{u_2u_1^{-1}}^{+138 \delta}$. By quasi-convexity of the axis, $[m_1',m_2']$ is in $C_{u_2u_1^{-1}}^{+148 \delta}$.  
As $p_2$ and $q_1$ lie in $[m_1',m_2']$, the points $u_2x_0$ and $u_1vx_0$ are in $C_{u_2u_1^{-1}}^{+190\delta}$. As $C_{u_2u_1^{-1}}^{+190 \delta}$ is invariant under the action of $u_2u_1^{-1}$ and of $u_2^{-1}u_1$, the points $u_1x_0$ and $u_2vx_0$ are in $C_{u_2u_1^{-1}}^{+190\delta}$.

 Finally, $u_1^{-1}C_{u_2u_1^{-1}}^{+190\delta}=C_{u_1^{-1}u_2}^{+190\delta}$. This concludes the proof.
  \end{proof}

\subsubsection{Periodic isometries}

In this section, we geometrise the idea of periodic words considered by Razborov \cite{razborov_product_2014} and Safin \cite{safin_powers_2011}, familiar in combinatorial small cancellation theory. We recall that $G$ acts on $X$, that the action is $(\kappa_0,N_0)$--acylindrical and that $$\nu=4N_0\frac{\kappa_0}{\rho_0}\hbox{ and } A=10^7  N_0^2 \frac{\kappa_0}{\rho_0} .$$ The main definition is:

\begin{df}[$E$--periodic] \label{D: periodic} Let $E$ be a maximal loxodromic subgroup.  An isometry $v\in G$ is  $E$--periodic at $x_0\in X$ if 
$$\hbox{$x_0$ and $vx_0$ are in $C_E^{+190 \delta}$ and } |vx_0-x_0|> 3 \nu [E] + A \delta + 10^7\delta.$$
 In this case, we call $E$ the \emph{period of $v$ at $x_0$}.
\end{df}

\begin{ex} Let $G$ be the free group acting on its Cayley graph. Then $\delta =0$ and we let $\kappa_0=\rho_0$, so that $\nu=4$. Let $u\not=1 \in G$ and let $E=E(u)$. We recall that $1\in C_u=C_E$ if, and only if, $u$ is cyclically reduced.  An element $v\in G$ is $E$--periodic at $1$ if, and only if, $u$ is cyclically reduced and decomposes as a (cyclically) reduced product $u=u_1u_2$  such that $v=u^ku_1$, where $k\geqslant 13$. 
\end{ex}

The small cancellation lemma \ref{L: small cancellation} implies:

\begin{lem}[Uniqueness of periods] \label{L: uniqueness of periods} If $v$ is periodic with periods $E_1$ and $E_2$ at $x_0$, then $E_1=E_2$. \qed 
\end{lem}

We also need the following notion, which is weaker than periodicity.
\begin{df}[$E$--left/right-periodic] \label{D: right-periodic} Let $E$ be a maximal loxodromic subgroup. 
An isometry $u$ is $E$--left-periodic at $x_0$ if 
$$\hbox{$x_0$ is in $C_E^{+190 \delta}$ and }  \diam([ux_0,x_0]\cap C_E^{+190 \delta})>3 \nu [E] + A \delta + 10^7\delta, $$
 and $[ux_0,x_0]\cap C_E^{+190 \delta}$ is the \emph{$E$--left-period} of $u$ at $x_0$.  

 If $u^{-1}$ is $E$--left-periodic at $x_0$, then $u$ is called \emph{$E$--right-periodic at $x_0$} and $[u^{-1}x_0,x_0]\cap C_E^{+190 \delta}$ is the \emph{$E$--right-period} of $u$ at $x_0$. 

\end{df}

An isometry can be right-periodic at $x_0$ with respect to two distinct maximal loxodromic subgroups. 
\begin{ex}
In the free group on generators $a$ and $b$ acting on its Cayley graph, $u_1=bba (ba^{13})^{13}$ and $u_2=(ba^{13})^{13}$ are both $\langle ba^{13} \rangle$-- and $\langle a \rangle$--right-periodic at the vertex representing the identity.
\end{ex}
But we still have the following.

\begin{lem}[Uniqueness of right-periods] \label{L: right-period}
If $E\not=E'$, if $[E]\geqslant [E']$ and if $u_1$ and $u_2$ are both $E$-- and $E'$--right-periodic at $x_0$, then their $E'$--right-periods are of Hausdorff distance at most $240 \delta$.
\end{lem}
 
\begin{proof}
With Lemma \ref{L: small cancellation} we deduce that $$\diam\left([u_1^{-1}x_0,x_0]\cap C_E^{+190 \delta}\right)>\diam\left(C_E^{+190 \delta}\cap C_{E'}^{+190 \delta}\right) + 10^6 \delta .$$ The same statement holds for $u_2$. 

Let $y_1$ and $y_2$ be the points in the $E$--right-period of $u_1$, or $u_2$ respectively, which are at distance $\diam\left(C_E^{+190 \delta}\cap C_{E'}^{+190 \delta}\right) + 10^6 \delta$ from $x_0$. By Lemma \ref{L: invariant cylinder}, $x_0$, $y_1$ and $y_2$ are in the $290\delta$--neighbourhood of a $200\delta$--local $1$--quasi--geodesic $L_g$. Hence, $|y_1-y_2|<2000\delta$. 

Therefore, if $y'_1$ is the point in the $E'$--right period of $u_1$ which is closest to $u_1^{-1}x_0$, then $y'_1$ is at distance at most $10\delta$ from $[u_2^{-1}x_0,x_0]$, and hence, the projection of $y'_1$ in $[u_2^{-1}x_0,x_0]$ is in $C_{E'}^{+200\delta}$. 

Now, by quasi-convexity of the invariant cylinders, cf. \cite[Lemme 2.2.2]{delzant_courbure_2008},
  $$\diam \left( C_E^{+200 \delta}\cap C_{E'}^{+200\delta}\right)< \diam \left( C_E^{+190 \delta}\cap C_{E'}^{+190 \delta}\right)+230\delta.$$ 
Thus, the projection of $y'_1$ onto $[u_2^{-1}x_0,x_0]$ is at distance at most 
 $\diam \left( C_E^{+190 \delta}\cap C_{E'}^{+190 \delta}\right)+230\delta$ from $x_0$. Hence, $y'_1$ is of distance at most $240\delta$ from the $E'$--right-period of $u_2$, and  $y'_2$,  the point in the $E'$--right period of $u_2$ which is closest to $u_2^{-1}x_0$, from the $E'$--right-period of $u_1$. This implies the claim.
\end{proof}

\subsubsection{Equations of reduced products}

Let us now consider a set of $n$ equations  $$u_0vw_0=u_1vw_1=\ldots =u_nvw_n,$$ where the products $u_0v$, $u_1v$, $\ldots$, $u_nv$ and $vw_0$, $vw_1$, $\ldots$, $vw_n$ are reduced at $x_0\in X$, and 
\begin{itemize}
\item $|u_0x_0-x_0|\leqslant |u_1x_0-x_0|\leqslant \ldots \leqslant |u_nx_0-x_0|$,
\item $26\delta <  |vx_0-x_0|$,
\item for all $0\leqslant i,j\leqslant n$, $|u_ix_0-u_jx_0|\leqslant |vx_0-x_0|$
\end{itemize}

\begin{lem}[Periodic isometries]\label{L: periodic} 
If $n\geqslant  5\nu$, and if, for all  $1\leqslant i < n$,  
$$|u_ix_0-u_{i+1}x_0|> A\delta +10^8\delta, $$
then all $u_i^{-1}u_{i+1}$ are hyperbolic and belong to the same  maximal loxodromic subgroup $E:=E(u_i^{-1}u_{i+1})$.

 Furthermore, $v$ is periodic with period $E$ at $x_0$, 
  $u_{n-1}$ and $u_{n}$ are $E$--right-periodic  at $x_0$ and their $E$--right-periods  are  at Hausdorff distance $>250 \delta$.
\end{lem}

This extends Safin's  \cite[Lemma 3]{safin_powers_2011} in the free group.

\begin{proof} Let us first assume that $n=5\nu$ and let $0\leqslant i < n$. Then $x_0$, $u_i^{-1}u_{i+1}x_0$ and $vx_0$ are in $C_{u_i^{-1}u_{i+1}}^{+190\delta}$  by Proposition \ref{P: reduced products}.  As $|u_ix_0-u_{i+1}x_0|>590 \delta$, the translation length $[u_i^{-1}u_{i+1}]>200 \delta$ by Lemma \ref{L: axis}. By Lemma \ref{L: invariant lines}, $u_i^{-1}u_{i+1}$ is hyperbolic.

In particular,  $E({u}_i^{-1}{u}_{i+1})$ is a loxodromic subgroup, and $x_0$, $u_i^{-1}u_{i+1}x_0$ and $vx_0$ are in  the $190 \delta$--neighbourhood of the $E({u}_i^{-1}{u}_{i+1})$--invariant cylinder, see Lemma \ref{L: invariant cylinder}. 

Let us fix $u_j$ such that  $|{u}_jx_0-{u}_{j+1}x_0|=\min_{0\leqslant i< n}|{u}_ix_0-{u}_{i+1}x_0|$.

As by Proposition \ref{P: reduced products} and Remark \ref{R: inversed triangle inequality}, 
\begin{align}
|{u}_{i}x_0-x_0| \geqslant |{u}_ix_0-{u}_{i-1}x_0| + |u_{i-1}x_0-x_0|- 48\delta, \label{reduced product nonsense}
 \end{align}
  for all $n\geqslant \iota\geqslant  2$ 
\begin{align}
|{u}_{\iota}x_0-{u}_0x_0| &\geqslant \sum_{l=0}^{\iota-1} |{u}_lx_0-{u}_{l+1}x_0|- 48\iota \delta  \label{crucial computation}\\
& \geqslant \iota \; |{u}_jx_0-{u}_{j+1}x_0| -48 \iota \delta. \nonumber 
\end{align}

Let us now set $E:=E ({u}_j^{-1}{u}_{j+1})$. We first show that $v$ is $E({u}_j^{-1}{u}_{j+1})$--periodic at $x_0$: indeed, by assumption, $|vx_0-x_0|\geqslant |{u}_nx_0-{u}_{0}x_0|$. By \eqref{crucial computation},
 \begin{align}
|vx_0-x_0| &\geqslant n \; |{u}_jx_0-{u}_{j+1}x_0| - 48n \delta. \nonumber \\
& = 5\nu  \; |{u}_jx_0-{u}_{j+1}x_0| - 240 \nu \delta >  3 \nu  \, |{u}_jx_0-{u}_{j+1}x_0|  + A \delta + 10^7 \delta. \nonumber
\end{align}
Therefore, $v$ is is $E({u}_j^{-1}{u}_{j+1})$--periodic at $x_0$.

Next, we fix $i$ as above and prove that ${u}_{i}^{-1}{u}_{i+1} E {u}_{i+1}^{-1}{u}_{i}=E$.

Recall that for $\alpha>10\delta$, $C_E^{+\alpha}$ is $10\delta$--quasi-convex, see Lemma \ref{L: invariant cylinder} and Remark \ref{R: neighbourhood quasi-convex}. By Proposition \ref{P: reduced products} and Lemma \ref{L: Gromov product 2}, $d({u}_i^{-1}{u}_{i+1} x_0,[x_0,vx_0])\leqslant 70\delta$, hence,  ${u}_i^{-1}{u}_{i+1} x_0 \in C_{E}^{+270\delta}$. Similarly, $d(vx_0,[{u}_i^{-1}{u}_{i+1} x_0,{u}_i^{-1}{u}_{i+1} vx_0])\leqslant 142\delta$ and $vx_0 \in {u}_i^{-1}{u}_{i+1}C_{E}^{+342\delta}$, which is the ${u}_{i}^{-1}{u}_{i+1}E {u}_{i+1}^{-1}{u}_{i}$--invariant cylinder.
By $10\delta$--quasi-convexity, $[{u}_i^{-1} {u}_{i+1} x_0,vx_0]\in C_E^{+352\delta}$ and $[{u}_i^{-1} {u}_{i+1} x_0,vx_0]\in {u}_i^{-1}{u}_{i+1} C_E^{+352\delta}$. 
 Now, 
$$|vx_0-{u}_i^{-1} {u}_{i+1} x_0|\geqslant |vx_0-x_0| - |{u}_{i+1}x_0-{u}_ix_0| \geqslant |{u}_{n}x_0-{u}_0x_0| - |{u}_{i+1}x_0-{u}_ix_0|.$$
 With \eqref{crucial computation}, we deduce that   
 \begin{align*}
 \diam \left( C_E^{+352\delta} \cap {u}_i^{-1}{u}_{i+1} C_E^{+352\delta}\right) & \geqslant (n-1) \; |{u}_jx_0-{u}_{j+1}x_0| -48 n\delta \\
& \geqslant 4 \nu  \, |{u}_jx_0-{u}_{j+1}x_0|  -240\nu \delta\\
&>  3 \nu  \, |{u}_jx_0-{u}_{j+1}x_0|  + A \delta + 10^7 \delta.
 \end{align*} 
By Lemma \ref{L: small cancellation}, ${u}_{i}^{-1}{u}_{i+1} E {u}_{i+1}^{-1}{u}_{i}=E$. 

Then, firstly, by maximality of $E$, ${u}_{i}^{-1}{u}_{i+1} \in E$, and, hence, $E=E({u}_{i}^{-1}{u}_{i+1} )$.

Secondly, ${u}_{0} E {{u}_{0}}^{-1}=\ldots = {u_i} E {u_i}^{-1}$, and hence, $ {u}_0x_0,\, \ldots,\, {u}_{i}x_0$ are in ${u}_{i} C_{E}^{+190 \delta}=C_{u_iEu_i^{-1}}^{+190\delta}$. Using \eqref{crucial computation} as before, we see that for all $i\geqslant 4\nu $, ${u}_{i}^{-1}$ is ${u}_{i} E {u_i}^{-1}$--left-periodic at ${u}_ix_0$. Hence, ${u}_{i}$ is $E$--right-periodic at $x_0$. 
By \eqref{reduced product nonsense}, $|u_{i+1}x_0-x_0|>|u_{i+1}x_0-u_ix_0|+|u_ix_0-x_0|-48\delta>|u_ix_0-x_0|+250\delta$, so that the respective $E$--right-periods have Hausdorff distance $>250\delta$.

This yields the assertion if $n=5\nu$. 
As the periods of $v$ at $x_0$ are unique (Lemma \ref{L: uniqueness of periods}), the result now follows for $n>5\nu$.
\end{proof}

\subsubsection{Equations of reduced products and acylindricity} \label{S: separation}

Recall that the action of $G$ on $X$ is $(\kappa_0,N_0)$--acylindrical. We keep the notations of the previous section. 
 
  \begin{prop}\label{P: periodic}  
 If  $n\geqslant 10^{12}N_0^4\frac{\kappa_0^2}{\rho_0^2}$ and $2500\kappa_0 \leqslant |vx_0-x_0|,$ then there is $i<n$ such that  $v$ is $E:=E(u_i^{-1}u_{n})$--periodic at $x_0$. In addition, $u_i$ and $u_n$ are $E$--right-periodic at $x_0$ and their $E$--right periods are of Hausdorff distance at least $250 \delta$.
\end{prop}
 
\begin{ex}  If $X$ is a tree, the proposition follows immediately from Lemma \ref{L: periodic}. More generally, let $\delta >0$ and the injectivity radius of the action larger than $10^9\delta$. Then we take $\kappa_0=\delta$, so that $A\leqslant 10^7$, and the proposition follows immediately by Lemma \ref{L: periodic}.
\end{ex}

In the general situation we assume that $2500\kappa_0 \leqslant |vx_0-x_0|$ and apply the acylindricity assumption as follows. 
It will be convenient to to fix $$n_1:=150N_0,$$ so that, by acylindricity, Lemma \ref{L: acylindrical}, at most $n_1$ group elements move $[x_0,vx_0]$ by at most $590\delta$. 

 \begin{lem} \label{L: large transl} Let $m>1$ and $n\geqslant mn_1$.  Let ${u'}_0:=u_0,\, {u'}_1:=u_{n_1},\, \ldots ,\, {u'}_{m}:=u_{mn_1}.$

 Then, for all $0\leqslant i \leqslant m$, $|{u'}_ix_0-{u'}_{i+1}x_0|>200\delta.$ 
\end{lem}

\begin{proof}
By Proposition \ref{P: reduced products}, for all $0\leqslant i\leqslant n$, $x_0$ and $vx_0$ are in $C_{u_i^{-1}u_{i+1}}^{+190 \delta}$. In particular, by Lemma \ref{L: axis},
 $|u_i^{-1}u_{i+1}vx_0-vx_0|\leqslant [u_i^{-1}u_{i+1}] +390 \delta$. Then $|u_i^{-1}u_{i+1}vx_0-vx_0|\leqslant |u_i^{-1}u_{i+1}x_0-x_0|+390\delta$. By assumption, $|vx_0-x_0|\geqslant 2500 \kappa_0$. We use acylindricity, see Lemma \ref{L: acylindrical}, to conclude the claim.
\end{proof}

\begin{lem} \label{L: large transl 2} Let $ p\geqslant 1$, $ m>1$ and $n\geqslant pmn_1$.  Let 
${u'}_0:=u_0,\, {u'}_1:=u_{mn_1},\, \ldots,\, {u'}_{p}:=u_{pmn_1}.$

Then, for all $1\leqslant j\leqslant p$, $|{u'}_jx_0- {u'}_{j+1}x_0|> 100 m \delta.$
 \end{lem}

 \begin{proof}  Let $g_0:=u_0,\, g_1:=u_{n_1},\, \ldots ,\, g_{m}:=u_{mn_1}, \ldots$. By Lemma \ref{L: large transl}, $|g_ix_0- g_{i+1}x_0|>200\delta$. 

We note that ${u'}_0=g_0$, ${u'}_1=g_{m}$, $\ldots$, ${u'}_{k}=g_{km}$.  Finally,  by Proposition \ref{P: reduced products} and Remark \ref{R: inversed triangle inequality}, 
$$|{u'}_jx_0- {u'}_{j+1}x_0|\geqslant \sum_{l=0}^{l=m-1} |g_{jm+l}x_0-g_{jm+l+1}x_0| - 48m\delta  > 100 m \delta.$$
  \end{proof}

\begin{proof}[Proof of Proposition \ref{P: periodic}] We note that  $n\geqslant  (5\nu) \cdot (A+10^6) \cdot 150N_0$. We recall that $n_1=150 N_0$ and let $m=A+10^6$. For $ p\geqslant 1$, let $ {u'}_{p}:=u_{pmn_1}$. By Lemma \ref{L: large transl 2}, $|{u'}_jx_0- {u'}_{j+1}x_0|> A\delta +10^8 \delta.$ By Lemma \ref{L: periodic}, we conclude Proposition \ref{P: periodic}. 
\end{proof}

\section{The case of bi-periodic sets} \label{S: Estimation}

Safin \cite[Section 3]{safin_powers_2011}, in the case of a free group, first proves Theorem \ref{IT: hyperbolic groups} in the case of a set $U$ that is contained in coset $Et$ of a cyclic subgroup $E$. In this case $E=\langle h\rangle$ for some hyperbolic element $h$, and  this implies that $U=\{h^{n_1}t,\, h^{n_2}t,\, \ldots\}$. We adapt this result in the case of an acylindrical action on $X$. 

We recall that the group $G$ acts on a $\delta$--hyperbolic space $X$, and that the action is $(\kappa_0,N_0)$--acylindrical.

\subsection{Bi-periodic sets}

Note that if $v$ is $E$--periodic at $x_0$ then $v^{-1}$ is $v^{-1}Ev$--periodic at $x_0$ because $v^{-1}C_E=C_{v^{-1}Ev}$.

\begin{df}[Bi-periodic set] \label{D: bi-periodic set} A finite subset $V\subset G$ is \emph{bi-periodic at $x_0\in X$}, if there are maximal loxodromic subgroups $E_1$ and $E_2$ such that for all $v\in V$ the isometry $v$ is $E_1$--periodic at $x_0$ and $v^{-1}$ is $E_2$--periodic at $x_0$. 
\end{df}  

\begin{ex} If $G$ is the free group, $a,b\in G$, $ab$ and $ba$ are reduced, then $\{(ab)^{13}a,(ab)^{14}a,(ab)^{15}a\}$ is a bi-periodic set at $1$.
\end{ex} 

By Lemma \ref{L: uniqueness of periods}, every bi-periodic set is contained in a coset of a maximal loxodromic subgroup:

\begin{lem}\label{L: bi-periodic set} Let $V\subset G$ be periodic at $x_0$. Then there is a maximal loxodromic subgroup $E$ and $t\in G$ such that $V\subset  Et$. 
\end{lem}
\begin{proof} Let $E_1$ be the period of the isometries $v\in V$, and let $E_2$ be the period of the  $v^{-1}\in V^{-1}$. 
Let $v_1,v_2\in V$. As $v_1$ is $E_1$--periodic at $x_0$, its inverse $v_1^{-1}$ is $v_1^{-1}E_1v_1$--periodic at $x_0$. By Lemma \ref{L: uniqueness of periods}, $v_1^{-1}E_1v_1=E_2$. Analogously, $E_2=v_2^{-1}E_1v_2$. Therefore, $E_1=v_1v_2^{-1}\cdot E_1 \cdot v_1^{-1}v_2$. By maximality, $v_1v_2^{-1} \in E_1$. We conclude that $V\subset E_1   v_2$.
\end{proof}

\subsection{Ping-pong with bi-periodic sets} \label{S: ping pong}

  If $V$ is a bi-periodic set, we adapt the arguments of \cite[\S 3]{delzant_sous_1991} to estimate the cardinality of $V^n$.

\begin{df}Let $t\in G$ and let $E$ be a maximal loxodromic subgroup. We say $t$ is \emph{$E$--reduced} at $x_0$ if 
$$x_0\in C_E \hbox{ and if } |tx_0-x_0|= \min_{e,f \in E} |(e\, t \,  f ) \cdot x_0 - x_0 |.$$
\end{df}

\begin{lem}\label{L: E reduction} Let $E$  be a maximal loxodromic subgroup and $x_0\in C_E$. If $t$ is $E$--reduced at $x_0\in C_E$, then for all $v\in E$ 
$$ (t^{\pm 1} x_0, vx_0)_{x_0}\leqslant [E]/2 + 1500\delta.$$
\end{lem}

\begin{proof}
By contradiction let us suppose that there is $v\in E$ such that $ (tx_0, vx_0)_{x_0}>  [E]/2 +1500\delta$. Without restriction we may further assume that $|vx_0-x_0|>[E]+1500\delta$.

 Let $g\in E$ realise $[E]$ and recall that $[E]>200\delta$.  By Lemma \ref{L: invariant cylinder}, $x_0$, $gx_0$ and $vx_0$ are in $ L_g^{+100\delta}$. Moreover, $|gx_0-x_0|\leqslant [E]+ 210 \delta$. By $11\delta$--quasi-convexity of $L_g^{+100 \delta}$ (Remark \ref{R: neighbourhood quasi-convex}), $[x_0,vx_0]\in L_g^{+111 \delta}$.  Without loss of generality - otherwise this holds for $g^{-1}$ - $d(gx_0,[x_0,vx_0])\leqslant 211 \delta$. 
  
  Let $p$ be a projection of $gx_0$ on a geodesic $[x_0,vx_0]$. Then $|p-x_0| \leqslant [E]+ 421 \delta$. 
 
   Let $q$ be a projection of $tx_0$ on $[x_0,vx_0]$. By Lemma \ref{L: Gromov product}, then 
   $$|q-x_0|\geqslant |tx_0-q|- |tx_0-x_0|\geqslant (x_0, vx_0)_{tx_0}-|tx_0-x_0|= (tx_0, vx_0)_{x_0}>[E]/2 +1500\delta.$$ If $p$ is on the left of $q$, $(tx_0,x_0)_{p}\leqslant (tx_0,x_0)_q \leqslant 4 \delta$ by Lemma \ref{L: thin triangles 2}. 
   
  Otherwise, as $(tx_0,x_0)_p \leqslant (tx_0,x_0)_q+|p-q| \leqslant 4\delta + |p-q|$ by Lemma \ref{L: thin triangles 2}, 
  \begin{align*}
  (tx_0,x_0)_p &  \leqslant  |p-x_0|-|q-x_0|  +4\delta <[E]/2-1075\delta .
  \end{align*}
Thus, 
$$(tx_0,x_0)_{gx_0}\leqslant (tx_0,x_0)_{p} + |p-gx_0| \leqslant [E]/2-864 \delta.$$ 
We conclude that 
  $$|tx_0-x_0|\geqslant |tx_0- gx_0| + |gx_0-x_0| - [E] +1728 \delta> |tx_0-gx_0|,$$ see Remark \ref{R: inversed triangle inequality}.  But then $|g^ {-1}tx_0 -x_0| < |tx_0-x_0|$, a contradiction.
  
  The argument for $t^{-1}$ is symmetric. 
\end{proof}

Recall that $\nu=4N_0\frac{\kappa_0}{\rho_0}$ and $A=10^7  N_0^2 \frac{\kappa_0}{\rho_0}.$ 
Let $E$   be a maximal loxodromic subgroup, let $t\not =1 $ be $E$--reduced at $x_0$, and let $$ a:= 3\nu\, [E] + A \delta + 10^5 \delta.$$

\begin{prop}\label{P: ping pong} Let $V\subset E$ be finite such that   
for all $v \in V$, $ |vx_0-x_0|\geqslant 10 a$, and  for all distinct $v,v' \in V$, $|vx_0-v'x_0|\geqslant 10 a$. Let $n$ be a natural number. Then $|(Vt)^n|\geqslant |V|^n$.
\end{prop}

In order to prove this proposition, let $v_1,\ldots, v_{n},\, w_1,\ldots w_n\in V$ such that $v_n\not= w_n$. We need to show that 
$$ v_1tv_2t\cdots v_nt\not= w_{1}tw_{2}t\cdots w_{n}t,$$
 or, equivalently, that
 $$(v_1)(tv_2)(tv_3)\cdots (t v_nw_{n}^{-1})(t^{-1}w_{n-1}^{-1}) \cdots (t^{-1} w_{1}^{-1})\not= 1.$$
We now define a sequence of points 
$ x_1:=v_1x_0$, $x_2:=(v_1)(tv_2)x_0$, $ \ldots ,$ $$  x_{2n-1}:=(v_1)(tv_2)(tv_3)\cdots (t v_nw_{n}^{-1})(t^{-1}w_{n-1}^{-1}) \cdots (t^{-1} w_{1}^{-1})x_0.$$
 In order to show that $x_0\not=x_{2n-1}$ we estimate the distance of $x_0$ to $x_{2n-1}$. To this end we verify that for all $1\leqslant i \leqslant 2n-3$, 
$$(x_i,x_{i+2})_{x_{i+1}}\leqslant \frac{1}{2}\min\{|x_i-x_{i+1}|,|x_{i+1}-x_{i+2}|\}-2\delta.$$ 
Then Lemma \ref{L: discrete quasi-geodesics} implies that $|x_{2n-1}-x_0|>0$, hence, the assertion of the Lemma. 

We note that for all $1\leqslant i \leqslant n$, the translation lengths $[w_i]>9a$, $[v_i]>9a$ and $ [v_nw_{n}^{-1}]>9a$: indeed, as for all $v,\, v'\in V$, $|vx_0-v'x_0|\geqslant 10a$ and $x_0\in C_E$, by Lemma \ref{L: invariant lines} 
$$[v'v^{-1}]=[v^{-1}v']\geqslant |vx_0-v'x_0| -210\delta >9a.$$

We first prove:
 \begin{lem}\label{L: ping pong 2} Let $h,\, g\in E$ such that $[g]>9a$ and $[h]>9a$. Then 
 \begin{enumerate}
 \item $(gt^{\pm 1}x_0,tht^{-1} (tx_0))_{x_0}< 2a$ and  $(gx_0,tht^{-1} (tx_0))_{x_0}< 2a$, 
 \item $(gt^{\pm 1}x_0,t^{-1}ht (t^{-1}x_0))_{x_0}< 2a$ and $(gx_0,t^{-1}ht (t^{-1}x_0))_{x_0}< 2a$.
\end{enumerate}  
 \end{lem}

\begin{proof} 
 We prove that $(gt^{-1}x_0,tht^{-1} (tx_0))_{x_0}< 2a.$ 
Let $p$ be a projection of $tx_0$, $q$ a projection of $y:=(tht^{-1})tx_0$ and $r$ a projection of $z:=gt^{-1}x_0$ on $[x_0,gx_0]$.  
By Lemma \ref{L: thin triangles 2}, and as $t$ is $E$--reduced at $x_0$ (Lemma \ref{L: E reduction}), 
\begin{align*}
|x_0-p|&\leqslant (tx_0,gx_0)_{x_0} +4\delta\leqslant [E]/2+1504 \delta< \frac13 a \hbox{ and}\\
|gx_0-r|&\leqslant (x_0,z)_{gx_0} +4\delta\leqslant [E]/2+1504 \delta<\frac13 a.
\end{align*}

We now show that $|p-q|\leqslant a$. Suppose that  $|p-q|>9\delta$. We first prove that $p$ and $q$ are  in $[tx_0,y]^{+9\delta}$. By Lemma \ref{L: thin triangles 2}, $(tx_0,q)_p\leqslant 4\delta$ and $(p,y)_q\leqslant 4\delta$. Then, by hyperbolicity, $\min\{(q,y)_p,(y,tx_0)_p\}\leqslant 5\delta$. As $(y,q)_p=|p-q|-(y,p)_q>5\delta$, then $(y,tx_0)_p\leqslant 5\delta$ and by Lemma \ref{L: Gromov product 2}, $d(p,[y,tx_0])\leqslant 9\delta.$ Inversing the role of $p$ and $q$  in this argument gives that $d(q,[y,tx_0])\leqslant 9 \delta$. This yields the claim.

Now, as $tC_E$ is $10\delta$--quasi-convex, see Lemma \ref{L: invariant cylinder}, $p$ and $q$ are in $tC_E^{+19\delta}$ and therefore in $C_E^{+19\delta}\cap tC_E^{+19\delta}$. Recall that $tC_E$ is the $t^{-1}Et$--invariant cylinder, and that $E\not= tEt^{-1}$ since $t$ is not in $E$. By Lemma \ref{L: small cancellation},  
$|p-q|\leqslant a.$

As $[g]>9a$ and $|gx_0-r|<\frac{1}{3} a$, $|r-x_0|>8a$, so that $r$ is on the right of $p$ and on the right of $q$. Moreover, $|r-q|>6a$.  
 By Lemma \ref{L: thin triangles 2}, $(y,r)_q\leqslant 4\delta$ and $(q,z)_r\leqslant 4\delta$. Then $(y,q)_r=|q-r|-(y,r)_q> |q-r|- 4\delta>5\delta$, and, by hyperbolicity, $\min\{(z,y)_r,(y,q)_r\}\leqslant 5\delta$, so that $(z,y)_r\leqslant 5\delta$.
Thus, using Remark \ref{R: inversed triangle inequality}, $$(z,y)_{x_0}\leqslant (r,q)_{x_0}+(z,y)_r+(r,y)_q\leqslant |x_0-p|+|p-q|+9\delta < 2a.$$
 
By symmetry the proofs of the remaining inequalities are analogous. 
\end{proof}

\begin{proof}[Proof of Proposition \ref{P: ping pong}]

  If $i<n $, let $g_i=v_i$, if $i>n$ let  $g_i=w_{i-n}^{-1}$, and let $g_n=v_nw_{n}^{-1}$. 
 Then for all $1\leqslant i \leqslant 2n-2$, $[g_i]>9a.$
    By  Lemma \ref{L: ping pong 2}, this gives that  $(x_i,x_{i+2})_{x_{i+1}}\leqslant 2a$.
 
As by Lemma \ref{L: E reduction} $(t^{\pm1}x_0,g_ix_0)_{x_0}\leqslant \frac13 a$, we have that (Remark \ref{R: inversed triangle inequality}) 
$$|x_i-x_{i+1}|=|t^{\pm 1}x_0 - g_ix_0|\geqslant |t^{\pm 1}x_0 -x_0| +[g_i] - 2a/3 >8a.$$
Therefore  $(x_i,x_{i+2})_{x_{i+1}}< \frac{1}{2}\min\{|x_i-x_{i+1}|,|x_{i+1}-x_{i+2}|\}-2\delta$, which we needed to prove. 

The assertion now follows from Lemma \ref{L: discrete quasi-geodesics}.
\end{proof}

\subsection{Separation of isometries in loxodromic groups}

Let $E$ be a maximal loxodromic subgroup, let $x_0\in C_E$ and let $V\subset E$ be finite. 

 By acylindricity (Lemma \ref{L: acylindrical}) and Lemma \ref{L: invariant cylinder}, there at most $N(1054/100) < 10^4 N_0$ elements $v$ of $E$ that satisfy 
$[v]\leqslant 844 \delta$.

\begin{lem}\label{L: separation} Let $r\geqslant1$ and assume that $|V|>12 \cdot 10^5 rN_0$. Then there is a subset $V_0 \subset V$ of cardinality at least $\frac{1}{6\cdot 10^5 rN_0} \; |V|$ such that 
$$ \hbox{ for all $v \in V_0$, $|vx_0-x_0|\geqslant r[E] $, and for all distinct $v,v'\in V_0$, $|vx_0-v'x_0|\geqslant  r[E] $.}$$
\end{lem}

\begin{proof} 
Let $x_0\in C_E$. Up to a factor of $1/2$, all elements of $V$ move $x_0$ closer to one limit point than to the other.
  Let  $V' \subseteq V$ be the maximal subset such that for all $v$, $v'\in V$, $[v]>844\delta$ and $[v^{-1}v']>844\delta$. By acylindricity, $|V'|\geqslant 1/(10^5 N_0) \; |V|-1$.

Let us enumerate the isometries of $V'$ such that 
$|v_ix_0-x_0|\leqslant |v_{i+1}x_0-x_0|$.  By Lemma \ref{L: invariant cylinder},  $x_0$, $v_ix_0$ and $v_{i+1}x_0$ are in the $100\delta$--neighbourhood of a bi-invariant line $L_g$. By $11\delta$--quasi-convexity of $L_g^{+100 \delta}$ (Remark \ref{R: neighbourhood quasi-convex}), $[x_0,v_{i+1}x_0]\subset L_g^{+111\delta}$, and, hence, $d(v_ix_0,[x_0,v_{i+1}x_0])\leqslant 211\delta$. By Lemma \ref{L: Gromov product}, $(x_0,v_{i+1}x_0)_{v_ix_0}\leqslant 211\delta$. Thus, $| v_{i+1}x_0 -x_0|\geqslant |v_ix_0-{v_{i+1}}x_0|+| v_{i}x_0 -x_0|-422\delta$, see Remark \ref{R: inversed triangle inequality}. In particular, $| v_{i+1}x_0 -x_0|\geqslant [v_i^{-1}v_{i+1}]+[v_{i-1}^{-1}v_i]+| v_{i-1}x_0 -x_0| -844 \delta \geqslant [E] +| v_{i-1}x_0 -x_0| $.

Let $V_0 \subset V'$ consist of every $2r$-th element of $V'$. The set $V_0$ is thus as required.
\end{proof}

\begin{lem}[Product growth of bi-periodic sets] \label{L: product growth of bi-periodic sets} Let $\gamma\geqslant 10^{14} N_0^3 \frac{\kappa_0}{\rho_0}$. 
Let  $t$ be an isometry that is not in $E$. If $|V|>4 \gamma$, then 
$$ | (Vt)^n|\geqslant \left(\frac{1}{\gamma}\, |V|\right)^{n}.$$
\end{lem}

\begin{proof} Let $x_0$ in $C_E$. We remark that $t=et'f$ such that $t'$ is not in $E$ and $E$--reduced at $x_0$.  Let $v_1,\ldots, v_{2n}\in V$ and let us suppose that $v_1tv_2t\cdots t v_n= v_{n+1}t \cdots t v_{2n}$. This is the case, if 
$$(fv_1e) t'(fv_2e)t'\cdots t' (fv_ne)= (fv_{n+1}e)t' \cdots t' (f v_{2n}e).$$
Then we apply Lemma \ref{L: separation} and obtain $V_0\subset fVe$ such that $|V_0|\geqslant {\gamma} |V|$ and so that Proposition \ref{P: ping pong}  can be applied to $t'$ and $V_0$. To see this note that $\frac{\gamma}{6\cdot 10^5} \geqslant 10 (3\nu + A + 10^5)N_0$. This yields the result.
\end{proof}

\begin{prop}\label{P: product growth bi-periodic sets}  Let $\gamma \geqslant 10^{14} N_0^3 \frac{\kappa_0}{\rho_0}$. 
 Let $U\subset G$ be finite and let $\beta>0$. Suppose $V\subset U$ with $|V|>4\beta |U|$ is periodic at $x_0$. If $U$ is not contained in an elementary subgroup, then $$|U^n|\geqslant  \left(\frac{\beta}{\gamma}\, |U|\right)^{[(n+1)/2]}.$$
\end{prop}

This proposition extends \cite[Section 3]{safin_powers_2011}.
\begin{proof} If $|U|\leqslant \frac{\gamma}{\beta}$ the inequality is true. Let $|U|>\frac{\gamma}{\beta}$. By Lemma \ref{L: bi-periodic set}, $V\subset Et$. If $t\not \in E$ we are done by Lemma \ref{L: product growth of bi-periodic sets}. Otherwise, there is $s\in U$ that is not in $E$. Then $|U^3|\geqslant | V s Vs|$ and we apply Lemma \ref{L: product growth of bi-periodic sets} to $Vs$.
\end{proof}

\section{Choice of the base point} \label{S: energy discussion} 

Let $U$ be a finite subset of $G$. In order to prove Theorems \ref{IT: hyperbolic groups}, \ref{IT: acylindrical trees} and \ref{IT: acylindrical hyperbolic}, as Safin \cite[Lemma 1]{safin_powers_2011} and Button \cite[Theorem 2.2]{button_explicit_2013}, we want to find two large subsets in $U$ whose products are reduced at a point $x_0\in X$. To this end we fix $x_0$ to minimise the  $\ell^1$--energy of $U$. 

\subsection{Energy and displacement}

Let us recall the definition of energy and displacement.

\begin{df}\label{D: energy} Let $U$ be a subset of $G$. The \emph{$\ell^1$--energy} of $U$ is $$E(U):= \inf_{x\in X} \frac{1}{|U|}\; \sum_{u\in U} |ux-x|.$$

Given a subset $U$, we fix $x_0 \in X$  such that $\frac{1}{|U|}\; \sum_{u\in U} |ux_0-x_0| \leqslant E(U)  + \delta$. 
\end{df}

The base point $x_0$ is fixed from now on.

\begin{df} Let $U$ be a subset of $G$. The \emph{displacement} of $U$ is $\lambda_0(U):=\max_{u\in U} |ux_0-x_0|$.
\end{df}

\begin{ex}\label{E: Energy and displacement}
If $G$ acts on a tree and $U\subset G$ is contained in an elliptic subgroup, then $U$ is fixing a point and $E(U)=0$ and $\lambda_0(U)=0$. 
\end{ex}
In general, $G$ acts on a $\delta$--hyperbolic space and we have:

\begin{prop}\label{P: Energy and displacement} If $U$ is contained in an elliptic subgroup, then $E(U)\leqslant 10 \delta$.

If, in addition, $|U|\geqslant 11N_0$, then $\lambda_0(U)\leqslant 2\kappa_0 + 12\delta$.
\end{prop}
\begin{proof} Let $U$ be contained in an elliptic subgroup $F$. The set the of $\delta$--almost fixed points $C_F:=\{x\in X \mid |fx-x|\leqslant 10\delta \hbox{ for all $f\in F$}\}$ is non-empty \cite[Proposition 2.36]{coulon_geometry_2014} and $8\delta$--quasi-convex \cite[Corollary 2.37]{coulon_geometry_2014}.

Then $E(U)\leqslant 10 \delta$. Moreover, if $|U|\geqslant 11N_0$, then $x_0 \in C_F^{+\kappa_0}$. Indeed, otherwise there are at most $N_0$--many $f\in F$ such that $|fx_0-x_0|\leqslant 11\delta$. Then $10\delta \geqslant E(U) > \frac{|U|-N_0}{|U|} 11\delta$, which would imply that $|U|< 11N_0$. 
 
  Finally, this implies that $\lambda_0(U)\leqslant 2\kappa_0 +12 \delta $. Indeed, if we take $p\in C_F$ such that  $ |p-x_0|\leqslant d(x_0,C_F)+\delta$, then $|fx_0-x_0|\leqslant 2d(x_0,C_F) + |fp-p| + 2\delta\leqslant 2\kappa_0 + 12 \delta$.
\end{proof}

\subsection{Concentrated and diffuse energy}
 As in the article of Button \cite{button_explicit_2013}, the proof splits into two cases, the case of concentrated energy (cf. \cite[Lemma 2.4]{button_explicit_2013}) and the case of diffuse energy (cf. \cite[Lemma 2.5 ff]{button_explicit_2013}). We recall that $U\subset G$ is a finite subset that is not contained in a loxodromic subgroup. 
If $G$ is a hyperbolic group acting properly and cocompactly on a $\delta$--hyperbolic space $X$ (Theorem \ref{IT: hyperbolic groups}), or if $G$ is a group acting  acylindrically on $X$ and $X$ is a tree (Theorem \ref{IT: acylindrical trees}), we let $d= 1$. If $G$ is a group acting acylindrically on $X$, $\delta >0$ and $X$ is a $\delta$--hyperbolic space (Theorem \ref{IT: acylindrical hyperbolic}), we let $d=\log_2(2|U|)$. 

We assume that $\lambda_0(U)>10^{14}d\kappa_0$. 

\begin{df}[Case 1 (Concentrated energy)]
The set $U$ is of \emph{concentrated energy} if 
for more than $1/4$ of the elements of $U\subset G$,  $$|ux_0-x_0|\leqslant 10^{10}d\kappa_0.$$
\end{df}
\begin{df}[Case 2 (Diffuse energy)] 
The set $U$ is of \emph{diffuse energy} if 
 for at least $3/4$ of the elements of $U\subset G$, $$|ux_0-x_0|>10^{10} d\kappa_0.$$
\end{df}

\begin{rem} The case of concentrated energy does not appear in free groups acting on trees as we can then set $\kappa_0=\rho_0/10^{10}$.  More generally, if the group $G$ acts acylindrically on a $\delta$--hyperbolic space and the injectivity radius of the action is larger than $10^{10}d\kappa_0$ then there are no subsets $U$ of concentrated energy.
\end{rem}

\begin{ex} \label{E: energy}  
Let $G$ split as a free product amalgamated over a finite subgroup. Then $d=1$ and let $\kappa_0= \rho_0/10^{14}$. Let $U\subset G$ be a finite subset that is not contained in an elliptic subgroup.  Then $\lambda_0(U) \geqslant  10^{14} \kappa_0$. 
 If $1/4$ of the elements of $U$ are conjugated into  free factor, then $U$ is of concentrated energy.
\end{ex}

\begin{ex}[Hyperbolic groups]
Let us assume that $G$ is hyperbolic and $X$ is a Cayley graph of $G$ with respect to a finite generating set that is $\delta$--hyperbolic.  Then  $d=1$. We let $\kappa_0=\delta$ and $b$ be the cardinality of the ball of radius $10^{14}\delta$. 
\label{L: small displacement hyperbolic} Thus,
\hbox{ if $|U|>b$, then $\lambda_0(U)> 10^{14} \kappa_0$. }
 
At this point let us note that if $|U|\leqslant b$ for some fixed $b>0$, then the product set estimates of Theorem \ref{IT: hyperbolic groups} are trivially satisfied whenever $\alpha \leqslant 1/b$. We can therefore assume that $|U|>b$ without restriction.
\end{ex}

\section{The case of concentrated energy}   
\label{S: small displacement}

If $G=A*B$ is a free product, $X$ its Bass-Serre tree, and $x_0\in X$ the fixed point of $A$, then $U \subset G$ is of concentrated energy if and only if there is a large subset $U_1\subset U$ that is contained in $A$, and $v\in U$ that is not in $A$. Then $U_1$ is fixing $x_0$ and $|vx_0-x_0| >0$. Thus $|U^3|\geqslant |U_1vU_1|\geqslant |U_1|^2$, cf. Button \cite[Lemma 2.4]{button_explicit_2013}. We adapt  Button's  strategy and use acylindricity to estimate the growth of the product sets $U^3$ and $U^n$ in the case of a set $U$ of concentrated energy.

We let $\delta >0$, and recall that $X$ is a $\delta$--hyperbolic space. The group $G$ acts on $X$ and the action assumed to be $(\kappa_0,N_0)$--acylindrical in the sense of Definition \ref{D: acylindrical}, and $\rho_0= \delta/N_0$.

\subsection{Third powers}

The next lemma is an application of Proposition \ref{P: reduced products} on reduced products. 

Let $U\subset G$. We fix $x_0$ and let $U_1\subset U$ consist of all $u\in U$ such that $|ux_0-x_0|\leqslant \kappa_0$. 
\begin{lem} \label{L: small displacement} There is a positive $c$ such that for all $v\in G$ with $|vx_0-x_0|\geqslant 10^4 \kappa_0$, $|U_1\, v \, U_1|\geqslant c^2 \; |U_1|^2.$
\end{lem}

In fact, in this section we take $c=\frac{1}{ 10^6} \cdot \frac{\rho_0}{\kappa_0}$.

\begin{proof} 
Let $u_1,w_1\in U_1$. Let $n>0$ such that for $n$-many $u_i\in U_1$ and $n$-many $w_i\in U_1$  
$$ u_1vw_1=u_ivw_i.$$
Note that $(u_1^{-1}x_0,vx_0)_{x_0}\leqslant \kappa_0$, $(u_i^{-1}x_0,vx_0)_{x_0}\leqslant \kappa_0$, $(v^{-1}x_0,w_1x_0)_{x_0}\leqslant \kappa_0$ and $(v^{-1}x_0,w_ix_0)_{x_0}\leqslant \kappa_0$. 
By the triangle inequality $|u_1{x_0}-u_i{x_0}| \leqslant 2\kappa_0$. By Proposition \ref{P: reduced products} (to apply the proposition, we note that $X$ is a $\kappa_0$--hyperbolic space), ${x_0}$ and $vx_0$ are in $C_{u_1^{-1}u_{i}}^{+190 \kappa_0}$. In particular, 
$| (u_1^{-1}u_i) vx_0-vx_0|\leqslant [u_1^{-1}u_i] + 390 \kappa_0 \leqslant 392 \kappa_0 .$

The result now follows  by Lemma \ref{L: acylindrical}. Indeed,  
 $\kappa_0(392 \kappa_0 /100\delta) < 2000 \kappa_0$ and $N(392 \kappa_0 /100\delta)\leqslant 150 N_0 \kappa_0/\delta $. Hence,  $n\leqslant 10^6 N_0 \kappa_0/\delta_0 = 10^6 \kappa_0/\rho_0$. 
\end{proof}

\subsection{Higher powers} In order to study higher powers of $U$, we keep the same notation for $c$ and $U_1$.

Let $v\in G$ such that  $|vx_0-x_0|\geqslant 10^4 \kappa_0$. Let  $m$  be the point on a geodesic $[x_0,vx_0]$ such that $$|m-x_0|=  500 \kappa_0.$$
We define $U_2\subset U_1$ to be the subset of maximal cardinality such that for all distinct $u,u'\in U_2$,
$$ |um-u'm| > 42\kappa_0.$$
Recall that $c=\frac{1}{ 10^6} \cdot \frac{\rho_0}{\kappa_0}$. 
\begin{lem}\label{L: higher powers one} Then  $|U_2|\geqslant c |U_1|$.
\end{lem}
 \begin{proof}
  By acylindricity, Lemma \ref{L: acylindrical}, for every $u\in U_2$ there are at most $\frac1c$--many $u'\in U_1$ such that $|um-u'm|\leqslant 100\kappa_0$. Hence, $|U_2| \frac{1}{c}\geqslant |U_1|$, so that $|U_2|\geqslant c |U_1|$. 
 \end{proof}

\begin{lem}\label{L: ping pong small displacement} Let $u,u' \in U_2$, and $u\not=u'$.  
\begin{enumerate}
\item Then $(uvx_0,u'vx_0)_{x_0}\leqslant 506\kappa_0$.
\item If $(v^{- 1}x_0,uvx_0)_{x_0}> 506\kappa_0$, then $(v^{-1}x_0,u'vx_0)_{x_0}\leqslant 506\kappa_0$.
\end{enumerate}
\end{lem}
 
\begin{proof} By Lemma \ref{L: thin triangles}, as $|ux_0-x_0|\leqslant \kappa_0$,  $d(um,[x_0,uvx_0])\leqslant 5\delta$. Similarly, $d(u'm,[x_0,u'vx_0])\leqslant 5\delta$.

To prove (1) assume by contradiction that 
$(uvx_0,u'vx_0)_{x_0}=(vx_0,u^{-1}u'vx_0)_{u^{-1}x_0}> 506\kappa_0.$ Let $r$ be a projection of $u'm$ on $[x_0,u'vx_0]$. As $d(u'm,[x_0,u'vx_0])\leqslant 5\delta$, $|x_0-r|\leqslant  506 \delta$. By Lemma \ref{L: thin triangles}, $d(r,[x_0,uvx_0])\leqslant 5\delta$. 
 Let $p$ be a projection of $u^{-1}u'm$ on $[x_0,vx_0]$. Then $|p-u^{-1}u'm|\leqslant d(u'm,[x_0,u'vx_0])+d(r,[x_0,uvx_0]) \leqslant 10 \delta$. 
Thus 
\begin{align*}
|um-u'm|& \leqslant \big|\,  |p-x_0| -|m-x_0| \, \big | + |p-u^{-1}u'm|\\
 & \leqslant  2|p-u^{-1}u'm| + |ux_0-x_0|+|u'x_0-x_0| \leqslant 23 \kappa_0, 
\end{align*}
which contradicts the assumption. This completes the proof of assertion (1).

Let us prove (2). Let $p$ be a projection of $um$ and $q$ a projection of $u'm$ on $[x_0,v^{-1}x_0]$.  
Assume by contradiction that $506\kappa_0< (v^{-1}x_0,u'vx_0)_{x_0}$. As in the proof of (1), we use Lemma \ref{L: thin triangles} to conclude that $|u'm-q|\leqslant d(u'm,[x_0,vx_0])+5\delta \leqslant 10 \delta$. Similarly, $|um-p|\leqslant 10\delta$.  
 Thus 
\begin{align*}
|um-u'm|& \leqslant \big|\,  |q-x_0| -|p-x_0| \, \big | + |p-um| + |q-u'm|\\
 & \leqslant  2|p-um| + 2 |q-u'm| + |ux_0-x_0|+|u'x_0-x_0| \leqslant 42 \kappa_0, 
\end{align*}
which contradicts the assumption. This completes the proof of assertion (2).
\end{proof} 

By Lemma \ref{L: ping pong small displacement}, up to removing one element from $U_2$, we now assume that for all distinct $u,u'\in U_2$,  
$$(v^{-1}ux_0,u'vx_0)_{x_0}\leqslant (v^{-1}x_0,u'vx_0)_{x_0} +|ux_0-x_0|\leqslant 507\kappa_0.$$
 
 We now prove:
\begin{lem}\label{eq: small displacement 3} Let $n\geqslant 0$ be a natural number. Then 
$
|(U_2\, v)^{n} \, |\geqslant  |U_2|^{n} 
$
\end{lem}

\begin{proof}
Let us fix $m\leqslant n$, $a_1, a_{2}, \ldots,\,a_{m},b_1, b_{2},\, \ldots,\,b_{m} \in U_2$ such that $a_1\not=b_1$, and show that   
  $$a_1va_{2} v \cdots a_mv\not=b_{1}vb_{2}v \cdots b_mv. $$
  In other words, we want to show that 
  $$(v^{-1}a_m^{-1})\cdots (v^{-1}a_1^{-1})\cdot(b_{1}v)\cdot(b_{2}v) \cdots (b_mv)\not=1. $$
To this end, we define a sequence of points $x_1:= (v^{-1}a_{m}^{-1}) x_0$, $x_2:= (v^{-1}a_{m}^{-1}) \cdot (v^{-1}a_{m-1}^{-1})x_0$, $\ldots,$ $x_{2m}:=(v^{-1}a_{m}^{-1}) \cdots 
(v^{-1}a_{1}^{-1}) \cdot(b_{1}v) \cdots ( b_{m}v)x_0.$ 
To estimate the distance of $x_0$ to $x_{2m}$, we will observe  that all Gromov products $(x_i,x_{i+2})_{i+1}\leqslant \frac12 \min\{|x_{i+1}-x_i|,|x_{i+1}-x_{i+2}|\} -2\delta$. 
 Then Lemma \ref{L: discrete quasi-geodesics} implies that $x_0\not=x_{2m}$.
 
 Now, $(x_i,x_{i+2})_{i+1}\leqslant 507\kappa_0$, see Lemma \ref{L: ping pong small displacement} and for all  $0\leqslant j< 2m$, $|x_j-x_{j+1}|\geqslant 9999 \kappa_0$. Indeed, for some $u\in U_2$, $|x_j-x_{j+1}|=|x_0-v^{\pm 1}u^{\pm 1}x_0|\geqslant |x_0-v^{\pm 1}x_0| - |u^{\pm 1}x_0-x_0|$, hence, the claim. 
\end{proof}
We conclude:
\begin{prop} \label{P: product growth in small displacement} 
If $|U_1|\geqslant 1/4 |U|$ and $\lambda_0(U)\geqslant  10^4  \kappa_0$, then
$$|U^n|\geqslant \left(\frac{c}{4} |U|\right)^{[(n+1)/2]}.$$ 
\end{prop}

\begin{proof} Assume without restriction that $|U_1|\geqslant 4/c$. Let $v\in U_1$ such that $|vx_0-x_0|\geqslant 10^4\kappa_0$.  If $n\leqslant 4$, the assertion follows from Lemma \ref{L: small displacement}.
If $n>4$, the assertion follows from the previous Lemma and Lemma \ref{L: higher powers one}.
\end{proof}

Proposition \ref{P: product growth in small displacement} completes the proof of Theorems \ref{IT: hyperbolic groups}, \ref{IT: acylindrical trees} and \ref{IT: acylindrical hyperbolic} in the case that $U$ is of concentrated energy:

\begin{cor} If $U\subset G$ is not contained in an elementary subgroup and of concentrated energy and $c=\frac{1}{ 10^{6}} \cdot \frac{\rho_0}{\kappa_0}$, then $$|U^n|\geqslant  \left(\frac{1}{10^{10}}\frac{1}{d}\frac{c}{4} |U|\right)^{[(n+1)/2]},$$
where $d=1$ in the case of hyperbolic groups and groups acting on trees, and $d=\log_2(2|U|)$ in the general case.
\end{cor}
\begin{proof} Let us first assume that $\delta>0$ and that $X$ is $\delta$--hyperbolic. Let $\kappa_0':=10^{10}d\kappa_0$. Note that the action of $G$ on $X$ is $(\kappa_0',N_0)$--acylindrical. The claim now follows from Proposition \ref{P: product growth in small displacement} .

If $X$ is a simplicial tree, we redefine $\delta:=\rho_0$ so that $\delta>0$ and $\kappa_0\geqslant \delta$. The result now follows as before.
\end{proof}

\section{The case of diffuse energy} \label{S: the case of diffuse energy}

 We now let $U\subset G$ be of diffuse energy and estimate the growth of $U^3$ and $U^n$ in this case. We recall that we fixed $x_0$ to minimise the energy $E(U)(x):= 1/|U|\; \sum_{u\in U} |ux-x|$ in Section \ref{E: Energy and displacement}. 
 
\subsection{Reduction Lemmas} \label{S: Energy and reduced products}

Recall that $\delta>0$, and that $X$ is a $\delta$--hyperbolic space or $\delta=0$ and $X$ is a simplicial tree. The group $G$ acts $(\kappa_0,N_0)$--acylindrical on $X$. 
 
 Recall that $d=1$ if $G$ is a hyperbolic group or if the group $G$ is acting on a  tree (Sections \ref{S: reduction bounded} and \ref{S: reduction trees} below), but $d= \log_2(2|U|)$ for groups acting on $\delta$--hyperbolic spaces (Section~\ref{S: reduction acylindrical}). 
 
As $U\subset G$ is of diffuse energy for at least $3/4$ of the  $u\in U$, $|ux_0- x_0|\geqslant 10^{10} d \kappa_0$. The aim of this section is to find two subsets $U_1$ and $U_2$ of $U$ of large cardinality such that all products $u_1u_2$ and $u_2u_1$ are reduced at $x_0$. 
This extends Safin's \cite[Lemma 1]{safin_powers_2011} in the case of a free group and Button's \cite[Theorem 2.2]{button_explicit_2013} in the case of a free product.

 In this section, after discussing some preliminaries in \ref{S: minimal energy}, we separately treat $3$ cases: first, we extend Safin's treatment of free groups \cite{safin_powers_2011} to groups acting acylindrically on a graph of bounded geometry (Section \ref{S: reduction bounded}).  
  Then we extend Button's treatment of free products \cite{button_explicit_2013}, first to groups acting acylindrically on trees (Section \ref{S: reduction trees}), and, finally, to groups with an acylindrical action on a general hyperbolic space (Section~\ref{S: reduction acylindrical}).

 \subsubsection{Preliminaries} \label{S: minimal energy}

  Let $S=S(x_0,1000 d\delta)$ denote the sphere of radius $1000 d\delta$ at $x_0$. Let $y$ and $z\in S$. Let 
 $$U_{y,z}:=\{u\in U\mid  |ux_0-x_0|\geqslant 4000 d\delta,\,  (x_0,ux_0)_{y}\leqslant d\delta \hbox{ and } (x_0,u^{-1}x_0)_{z}\leqslant d\delta \}.$$
  This is the set of isometries $u\in U$ such that $y$ is on a geodesic segment $[x_0,ux_0]$, such that $z$ is on a geodesic segment $[x_0,u^{-1}x_0]$ and such that $|y-uz|\geqslant 2000d\delta$.

\begin{lem}\label{L: thin triangles 3}
 If   $|z_0-y_1|> 6d\delta$ and $|y_0-z_1|> 6d\delta$, then 
  $$\left( U_{y_0,z_0}^{-1}x_0, U_{y_1,z_1}x_0 \right)_{x_0}\leqslant 1000 d\delta \hbox{ and } \left( U_{y_0,z_0 }x_0, U_{y_1,z_1}^{-1}x_0 \right)_{x_0}\leqslant 1000 d\delta.$$
\end{lem}

\begin{proof}
By contradiction assume that $\left( U_{y_0,z_0}^{-1}x_0, U_{y_1,z_1}x_0 \right)_{x_0}> 1000 d\delta $.  Let $u_0\in U_{y_0,z_0}$ and $u_1\in U_{y_1,z_1}$. By definition of hyperbolicity, 
$(z_0,y_1)_{x_0}\geqslant \min\{(z_0,u_0^{-1}x_0)_{x_0},(u_0^{-1}x_0,u_1x_0)_{x_0},(u_1x_0,y_1)_{x_0}\}-2d\delta.$

As $1000d\delta \geqslant (z_0,u_0^{-1}x_0)_{x_0}= |z_0-x_0|-  (x_0,u_0^{-1}x_0)_{z_0} \geqslant 999d\delta $,  and, similarly, $1000d\delta \geqslant (y_1,u_1x_0)_{x_0} \geqslant 999d\delta $, 
$(z_0,y_1)_{x_0}=1000d\delta -\frac12 |z_0-y_1|\geqslant 997 d\delta. $ 

Therefore, $6d\delta \geqslant |z_0-y_1|$, a contradiction. This proves the first assertion of the Lemma, the proof for the second assertion is symmetric.
\end{proof}

\begin{lem}[Minimal energy]\label{L: minimal energy} Suppose that $\delta>0$. 
 Let $y_0\in S(x_0,1000d\delta)$ and let  $B(y_0,100d\delta)$ be the ball of radius $100d\delta$ at $y_0$, then 
 $$ \left| \bigcup_{y,z\in B(y_0,100d\delta)} U_{y,z}\; \right|  \leqslant 2/3\;  |U|.$$
\end{lem}

This lemma corresponds to Case 2 in Safin's proof of \cite[Lemma 1]{safin_powers_2011}.

\begin{proof} 
By contradiction, let us assume that  
$$ \left| \bigcup_{y,z\in B(y_0,100d\delta)} U_{y,z}\; \right|  >2/3\;  |U|.$$

First, let $y,z\in B(y_0,100d\delta)$ and let  $u\in U_{y,z}$.  
 We prove that $|u y_0-y_0|\leqslant |ux_0-x_0|-1594 d\delta$.

As $|y-y_0|\leqslant 100 d\delta$ and $|z-y_0|\leqslant 100d \delta$,
\begin{align*}
(x_0,ux_0)_{y_0}&\leqslant (ux_0,x_0)_{y} +|y-y_0| \leqslant 101 d\delta \hbox{ and }\\
 (x_0,ux_0)_{uy_0}&=(x_0,u^{-1}x_0)_{y_0}\leqslant (u^{-1}x_0,x_0)_z +|z-y_0|\leqslant 101 d\delta.
\end{align*}
Then, by Remark \ref{R: inversed triangle inequality},
$$|ux_0-x_0| \geqslant |uy_0-x_0| + |y_0-x_0| - 202d \delta = |uy_0-x_0| + 798d \delta .$$ 

By definition of hyperbolicity, $(x_0,ux_0)_{y_0}\geqslant \min\{(x_0,uy_0)_{y_0} , (uy_0,ux_0)_{y_0}\} -d\delta$. 

We prove that  $(uy_0,ux_0)_{y_0}\geqslant 102 d \delta$. Otherwise, as $|y_0-x_0|= 1000d\delta$, $(uy_0,ux_0)_{y_0}= \frac12 (|uy_0-y_0|+|ux_0-y_0|-1000 d\delta)\geqslant\frac12 |ux_0-y_0| - 500 d\delta,$ and, hence, $|ux_0-y_0|\leqslant 1204 d\delta$. Then $|ux_0-x_0|\leqslant |ux_0-y_0| + |y_0-x_0|\leqslant 2204 d\delta $, a contradiction as $|ux_0-x_0|>4000d\delta$. 
 
Thus $(x_0,uy_0)_{y_0} \leqslant 102 d \delta$ and, see Remark \ref{R: inversed triangle inequality}, 
$$|uy_0-x_0|\geqslant |uy_0-y_0| + |x_0-y_0| -204d\delta =  |uy_0-y_0| + 796 d\delta.$$ 
We conclude that $|ux_0-x_0| \geqslant |uy_0-y_0| + 1594 d\delta ,$ which implies the claim.

Otherwise, if $u\not\in \bigcup_{y,z\in B(y_0,100d\delta)} U_{y,z}$, then
\begin{align*}
|uy_0-y_0|& \leqslant  |u y_0-ux_0| + |ux_0-x_0| +  |x_0-y_0| \\
& =  2 |x_0-y_0|+|ux_0-x_0|\leqslant |ux_0-x_0|+ 2000d\delta.
\end{align*}

Then 
$
\frac{1}{|U|} \sum_{u\in U} |uy_0-y_0| \leqslant E(U)-\frac23\cdot 1594 d\delta+\frac13\, \cdot 2000 d\delta +\delta  \leqslant E(U)-395d\delta.
$ 
This implies that $x_0$ was not minimising for the energy of $U$, a contradiction.
\end{proof}

\subsubsection{Reduction lemma 1: the case of graphs of bounded geometry}\label{S: reduction bounded} In this section, let us assume further that $X$ is a graph of uniformly bounded valence, for instance $X$ is a Cayley graph, or a regular $d$-valenced tree.  Let $b$ be a uniform bound on the number of vertices in a closed ball of radius $1000\delta$ in $X$.

The following lemma is a geometrization of Safin's Lemma 1 \cite{safin_powers_2011}, except that in his case the assumption is  trivially satisfied.

\begin{lem}[Reduction]\label{L: reduction}
If at most $1/4$ of the isometries  $u\in U$ have displacement $|ux_0- x_0|\leqslant 10^{10}\kappa_0$, then there are $U_1,U_2\subset U$ of  cardinalities  at least $\frac1{100 {b}^2}\; |U|$  such that 
$$ \hbox{$(U_1^{-1}x_0, U_2x_0)_{x_0}\leqslant 1000 \delta$ and $(U_2^{-1}x_0 ,U_1x_0)_{x_0} \leqslant 1000 \delta$.}$$
 
In addition, $|u_1x_0-x_0|\geqslant 10^{10}\kappa_0$ and $|u_2x_0-x_0|\geqslant 10^{10}\kappa_0$ for all $u_1\in U_1$ and all $u_2\in U_2$.
\end{lem}

 \begin{proof}  
 If $|z-y|>6\delta$ and $|U_{z,y}|>\frac{1}{100b^2} |U|$ we set $U_1=U_2=U_{z,y}$. Let us assume that there are no such $z,y\in S$. Then 
$$\left |  \; \bigcup_{\substack{(z,y)\in S\times S\\ |z-y|>6\delta}} U_{z,y}\; \right | \leqslant \sum_{ \substack{(z,y)\in S\times S\\ |z-y|>6\delta}} |U_{z,y}| \leqslant \frac1{100} |U| $$ 
and 
$$\left |  \; \bigcup_{\substack{(z,y)\in S\times S\\|z-y|\leqslant 6\delta}} U_{z,y}\; \right | \geqslant |U|-\frac1{100}|U|-\frac1{4}|U|\geqslant \frac{74}{100} |U|. $$

Hence, there is $z_0,y_0\in S$  of distance $|z_0-y_0|\leqslant 6\delta$ such that $|U_{z_0,y_0}|\geqslant \frac{74}{100{b}^2}|U|$. 
Indeed, otherwise 
$$\left |  \; \bigcup_{\substack{(z,y)\in S\times S\\ |z-y|\leqslant 6\delta}} U_{z,y}\; \right | \leqslant \sum_{\substack{(z,y)\in S\times S\\ |z-y|\leqslant 6\delta}} |U_{z,y}| < \frac{74}{100}\;  |U|$$

Let us fix $z_0,y_0\in S$  with $|z_0-y_0|\leqslant 6\delta$ such that $|U_{z_0,y_0}|\geqslant \frac{74}{100{b}^2} |U|$. 
If there are $z_1,y_1$ with $|z_1-y_1|\leqslant 6\delta$, $|z_1-z_0|>100 \delta$ and $|y_1-y_0|>100\delta$ and such that $|U_{z_1,y_1}|\geqslant \frac1{100 {b}^2}\; |U|$, then we set $U_1=U_{z_0,y_0}$ and $U_2=U_{z_1,y_1}$. 

 Let us show that such a set $U_2$ exists. 
Let us assume that  there are no $z_1,y_1$ with $|z_1-y_1|\leqslant 6\delta$, $|z_1-z_0|>100 \delta$ and $|y_1-y_0|>100\delta$ such that $|U_{z_1,y_1}|\geqslant \frac1{100 {b}^2}\; |U|$. Then $$ \left| \bigcup_{\substack{(z,y)\in S\times S\\ y,z\in B(y_0,100\delta)}} U_{y,z}\; \right| \geqslant  \frac{74}{100}\; |U|  - \frac1{100} \; |U| >\frac23\;  |U|.$$
 This contradicts the assumption of minimal energy at $x_0$ by Lemma \ref{L: minimal energy} (we can without restriction assume that $\delta>0$). 
 Thus, we have found subsets $U_1$ and $U_2$ with the desired properties, see Lemma \ref{L: thin triangles 3} . 
\end{proof}

\subsubsection{Reduction lemma 2: the case of trees} \label{S: reduction trees}

We still denote by $x_0$  a minimiser for the energy of $U$ and let $r \leqslant 1/4 \kappa_0$  be a fixed positive constant.

The proof of the following key lemma is due to Button \cite[Theorem 2.2]{button_explicit_2013}, formulated in a slightly different language. 

\begin{lem}[Reduction]\label{L: reduction trees}  
If at most $1/4$ of the isometries  $u\in U$ have displacement $|ux_0- x_0|\leqslant \kappa_0 $, then there are $U_1,U_2\subset U$ of  cardinalities  at least $\frac{1}{100} |U|$  such that 
$$\hbox{$(U_1^{-1}x_0,U_2 x_0)_{x_0}\leqslant r$ and $(U_2^{-1}x_0,U_1 x_0)_{x_0}\leqslant r$.}$$
 
In addition, $|u_1x_0-x_0|\geqslant \kappa_0 $ and $|u_2x_0-x_0|\geqslant \kappa_0 $ for all $u_1\in U_1$ and all $u_2\in U_2$.
\end{lem}

Let $S:=S(x_0,r)$ be the sphere of radius $r$ at $x_0$. For $A,B\subset S$ we denote by $U_{A,B}$ the set of isometries $u\in U$ such that $ |ux_0-x_0|\geqslant 4r$ and such that there are $a\in A$ and $b\in B$ with 
$$(x_0,ux_0)_{a}=0 \hbox{ and } (x_0,u^{-1}x_0)_{b}=0.$$ 
  
In other words, $U_{A,B}$ is the set of $u\in U$ such that $ |ux_0-x_0|\geqslant 4r$ and such that a geodesic segment from $x_0$ to $ux_0$ meets $A$ and such that a geodesic segment from $x_0$ to $u^{-1}x_0$ meets $B$ respectively.  

As we are working in a tree, the following lemma is immediate.

\begin{lem} Let $A \subset S$ and $B\subset S\setminus A$. Then  
 $$\left( \; U_{\star ,B}^{-1}x_0,U_{A,\star }x_0\; \right)_{x_0}\leqslant r.
$$ \qed
\end{lem}

Here we use $\star$ as a placeholder for a subset of $S$. 

\medskip

Let $$S':= \{ s\in S \mid \hbox{ there is $u\in U$ with $(ux_0,x_0)_{s}=0$ or $(u^{-1}x_0,x_0)_{s}=0$}\}.$$
In particular, $S'$ is a finite subset of $S$.

\medskip

Let $A\subset S'$ and let $B:= S'\setminus A$.

 Let us fix a point $a\in A$. Then we write  $A':= A \setminus \{a\}$, $B':= B \cup \{a\}$. 

\begin{rem} 
$|U_{A,A}| \geqslant |U_{A',A'}|$ and  $|U_{B,B}| \leqslant |U_{B',B'}|$.
\end{rem}

Lemma \ref{L: minimal energy} on minimal energy gives that $|U_{a,a}| \leqslant 2/3\; |U|$. A counting argument that takes this into account shows:

\begin{lem} \label{L: button}
If 
\begin{enumerate}
\item $|U_{A,B}|\leqslant \frac{1}{100} |U|$, $|U_{B,A}|\leqslant \frac{1}{100}   |U|$ and $|U_{B,B}|\leqslant \frac{1}{100} |U|$, and if
\item $|U_{A',B'}|\leqslant \frac{1}{100}  |U|$, $|U_{B',A'}|\leqslant \frac{1}{100} |U|$ and $|U_{B',B'}|> \frac{1}{100} |U|$,
\end{enumerate}
then $|U_{A',A'}|> \frac{1}{100}  |U|$.
\end{lem}
\begin{proof} By contradiction, let us assume that $$|U_{A',A'}|\leqslant \frac{1}{100} |U|.$$

By (1), $|U_{A,A}|> \frac{97}{100} |U| - \frac{1}{4} |U|$. Hence,
$$|U_{A,A} \setminus U_{A',A'}| \geqslant \frac{96}{100} |U| - \frac{1}{4} |U|.$$

We note that $A \times A \setminus A'\times A' = A' \times \{ a \} \cup \{ a \} \times A' \cup \{ a \} \times \{ a \}$. Therefore 
\begin{align*}
|U_{A,A} \setminus U_{A',A'}| &\leqslant |U_{A',  a }| + |U_{ a, A'}| + |U_{a,a}|\\
  &\leqslant |U_{A', B' }| + |U_{ B', A'}| + |U_{a,a}| \leqslant \frac{2}{100} |U| + |U_{a,a}|
\end{align*}
Hence, $|U_{a,a}| \geqslant \frac{94}{100} |U|- \frac{1}{4} |U|>\frac{2}{3}|U|$, a contradiction to Lemma \ref{L: minimal energy} on minimal energy (note that $X$ is $\frac{r}{1000}$--hyperbolic, so that the lemma applies with $\delta=\frac{r}{1000}$). 
\end{proof}

\begin{proof}[Proof of Reduction Lemma] We define subsets $A^{(n)},B^{(n)}$ of $S'$ by recursion:

\begin{itemize}
\item
Let $A^{(0)}:= S'$, and let $B^{(0)}$ be the empty set.
\item
For $n>0$, let $a_n\in A^{(n-1)}$ and let $A^{(n)}:= A^{(n-1)} \setminus \{a_n\}$, $B^{(n)}:= B^{(n-1)} \cup \{a_n\}$.
\end{itemize}

 \medskip

If  there is $n$ such that 
$|U_{A^{(n)},B^{(n)}}|> \frac{1}{100} |U|$ or  $|U_{B^{(n)},A^{(n)}}|> \frac{1}{100} |U|$, we set $U_1=U_2=U_{A^{(n)},B^{(n)}}$, or $U_{B^{(n)},A^{(n)}}$ respectively. 

\medskip

We assume that this is not the case for all $n\geqslant 0$. By the remark, there is $n >0$ such that
$$\hbox{ $|U_{B^{(n-1)},B^{(n-1)}}|\leqslant \frac{1}{100}  |U|$ and $|U_{B^{(n)},B^{(n)}}|> \frac{1}{100}  |U|$. }$$
By the lemma, $$|U_{A^{(n)},A^{(n)}}|> \frac{1}{100}  |U|.$$
 We then set $U_1=U_{A^{(n)},A^{(n)}}$ and $U_2=U_{B^{(n)},B^{(n)}}$. 
\end{proof}

\subsubsection{Reduction lemma 3: the case of hyperbolic spaces} \label{S: reduction acylindrical}
 Let $\delta>0$ and recall that $x_0$ is a minimiser for the energy of $U$. The idea is to use Gromov's tree approximation lemma,  Theorem \ref{T: tree approximation}, to adapt Button's argument.

\begin{lem}[Reduction]\label{L: reduction acylindrical}

If at most $1/4$ of the isometries  $u\in U$ have displacement $|ux_0- x_0|\leqslant  10^{10}\kappa_0\log_2 (2|U|)  $, then there are $U_1,U_2\subset U$ of  cardinalities  at least $\frac{1}{100} |U|$  such that 
$$\hbox{$(U_1^{-1}x_0,U_2 x_0)_{x_0}\leqslant 1000 \log_2 (2|U|) \delta $ and $(U_2^{-1}x_0,U_1 x_0)_{x_0}\leqslant 1000 \log_2 (2|U|) \delta $.}$$
 
In addition, for all $u_1\in U_1$ and all $u_2\in U_2$ 
$$\hbox{$|u_1x_0-x_0|\geqslant 10^{10}\kappa_0\log_2 (2|U|) $ and $|u_2x_0-x_0|\geqslant 10^{10}\kappa_0\log_2 (2|U|) $.}$$
\end{lem}

 Let $$r:=1000  \log_2 (2|U|) \delta$$ and let $S:=S\left( x_0,r \right)$ be the sphere of radius $1000 \log_2 (2|U|) \delta$ at $x_0$. 

Let $T$ be a tree approximating $x_0 \cup Ux_0 \cup U^{-1}x_0$, and let $f$ be a map from $\bigcup_{x\in Ux_0\cup U^{-1}x_0} [x_0,x]$ to $T$  that is given by  Theorem \ref{T: tree approximation}. We recall that $f$ is an isometry on each geodesic $[x_0,x]$. In addition, for all $y,z\in \bigcup_{x\in Ux_0\cup U^{-1}x_0} [x_0,x]$  we have $$|y-z| -4\delta \log_2 (2|U|) < |f(y)-f(z)| \leqslant |y-z|.$$

Let us denote by $\overline{S}$ the sphere of radius $r$ at $f(x_0)$ in $T$. It is the image of $S\bigcap \bigcup_{x\in Ux_0\cup U^{-1}x_0} [x_0,x]$ under $f$. 
For $A,B\subset \overline{S}$ we denote by  $\overline{U_{A,B}}$ the set of isometries $u\in U$ such that $|ux_0-x_0|\geqslant 4000 \log_2(2|U|) \delta$ and such that there are $a\in A$, $b\in B$ with 
$$\hbox{  $(f(x_0),f(ux_0))_{a}=0$ and $(f(x_0),f(u^{-1}x_0))_{b}=0$}.$$

\begin{lem} Let $A \subset \overline{S}$ and $B\subset \overline{S}\setminus A$. Then  
 $$\left( \; \overline{U_{\star ,B}}^{-1}x_0,\overline{U_{A,\star }}x_0\; \right)_{x_0}\leqslant r.
$$ 
\end{lem}
Here we use $\star$ as a placeholder for a subset of $S'$.
\begin{proof}
We let $u_0\in \overline{U_{A,\star }}$ and $u_1\in \overline{U_{\star ,B}}x_0$. 
By definition, there are $a\in A$ and $b\in B$ such that 
$$\left( f(x_0),f(u_0x_0) \right)_a=0 \hbox{ and } \left( f(x_0),f(u_1^{-1}x_0)\right)_b=0.$$
In $T$ we hence have that $$\left(f(u_0x_0),f(u_1^{-1}x_0)\right)_{f(x_0)}\leqslant r.$$
 The tree approximation gives that 
 $$\left(u_0x_0,u_1^{-1}x_0\right)_{x_0}\leqslant \left(f(u_0x_0),f(u_1^{-1}x_0)\right)_{f(x_0)}.$$
This implies our claim.
\end{proof}

Let $A\subset \overline{S}$, let $B:= \overline{S}\setminus A$. 

 Let us now fix $a\in A$ and write  $A':= A \setminus \{a\}$, $B':= B \cup \{a\}$. 

\begin{rem} 
$|\overline{U_{A,A}}| \geqslant |\overline{U_{A',A'}}|$ and  $|\overline{U_{B,B}}| \leqslant |\overline{U_{B',B'}}|$.
\end{rem}

\begin{lem} 
If 
\begin{enumerate}
\item $|\overline{U_{A,B}}|\leqslant 1/100 \; |U|$, $|\overline{U_{B,A}}|\leqslant 1/100 \;  |U|$ and $|\overline{U_{B,B}}|\leqslant 1/100 \;  |U|$, and if
\item $|\overline{U_{A',B'}}|\leqslant 1/100 \;  |U|$, $|\overline{U_{B',A'}}|\leqslant 1/100 \;  |U|$ and $|\overline{U_{B',B'}}|> 1/100 \;  |U|$,
\end{enumerate}
then $|\overline{U_{A',A'}}|> 1/100 \;  |U|$.
\end{lem}
\begin{proof} By contradiction, let us assume that $$|\overline{U_{A',A'}}|\leqslant \frac{1}{100} |U|.$$
Exactly as in the case of trees, see Lemma \ref{L: button}, we conclude that $|\overline{U_{a,a}}| > \frac{2}{3}|U|$. Now, let $y_0 \in f^{-1}(a)$.  Then, by tree approximation $f^{-1}(a)\subset B(y_0,4\delta \log_2(2|U|))\cap S$. Let us also note that 
$$\overline{U_{a,a}}\subseteq \bigcup_{y,z\in B(y_0,4\delta \log_2(2|U|))} U_{y,z}.$$
Hence, $\left| \bigcup_{y,z\in B(y_0,4\delta \log_2(2|U|))} U_{y,z} \right| > \frac{2}{3} |U|$. 
 This contradicts Lemma \ref{L: minimal energy} on minimal energy.
\end{proof}

We conclude the proof of Lemma \ref{L: reduction acylindrical} by the same argument used for Lemma \ref{L: reduction trees} (where $\overline{S}$ plays the role of $S'$ and sets $\overline{U_{\star,\star'}}$ the role of $U_{\star,\star'}$). \qed

\subsection{Reduction to bi-periodic sets} \label{S: Periodicity}

If $U$ is of diffuse energy in $G$ and not contained in a loxodromic subgroup, we let $U_1$ and $U_2$ be the sets given by the reduction lemmas in the previous section and define $W_l=(U_1U_2)^{l-2}U_1$, where $l=[(n+1)/2]$, so that $|U^n|\geqslant |U_1U_2W_l|$. 
 As Safin \cite{safin_powers_2011}, we will show that there is $c>0$ such that $|U^n|\geqslant |U_1U_2W_l| \geqslant c |U_1||W_l|$ unless $U_2$ is bi-periodic at $x_0$, that is, all elements of $U_2$ are periodic with same period and same tale. The case of bi-periodic sets was already treated in Section \ref{S: Estimation}.  
This will complete the proofs of Theorems \ref{IT: hyperbolic groups}, \ref{IT: acylindrical trees} and \ref{IT: acylindrical hyperbolic}.

\subsubsection{Equations of reduced products, revisited} We recall that $X$ is $\delta$--hyperbolic, and that the action of $G$ on $X$ is $(\kappa_0,N_0)$--acylindrical in the sense of Definition \ref{D: acylindrical}. Moreover, if $\delta>0$, then $\rho_0=\delta/N_0$, if $\delta=0$ then $X$ is a simplicial tree of edge length $\rho_0$. 

Let us now fix three finite subsets $U_1$, $U_2$ and $W$ of $G$ and a point $x_0\in X$ such that, firstly, the products in the sets $U_1U_2$ and $U_2W$ are reduced at $x_0$, and secondly, for all $u\in U_1$ and all $v\in U_2$,  $2500 \kappa_0 \leqslant |ux_0-x_0|\leqslant |vx_0-x_0|$.

\begin{lem}\label{ML: counting}    Let $c\geqslant  10^{12}N_0^4 \frac{\kappa_0^2}{\rho_0^2}$.
 If $|U_1|>2c$ and $|W|>2c$ then, for all $v\in U_2$,  
 $$|U_1vW|> \frac1{2c} \; |U_1||W|$$ 
 unless there are maximal loxodromic subgroups $E$ and $u,u' \in U_1 $ and $w,w' \in W$ such that 
 \begin{enumerate}
 \item $v$ is $E$--periodic at $x_0$
\item $u$ and $u'$ are $E$--right-periodic at $x_0$ and $w$ and $w'$ are $v^{-1}Ev$--left-periodic at $x_0$.
 \end{enumerate}
The $E$--right-periods of $u$ and $u'$ at $x_0$, and the $vEv^{-1}$--left-periods of $w$ and $w'$ at $x_0$ respectively, are of Hausdorff distance  $>250 \delta$.
\end{lem} 
If $|U_1vW|\leqslant \frac1{2c}  |U_1||W|$, then there are distinct $u_0,\ldots,u_{c}\in U_1$ and $w_0,$ $\ldots,$ $w_{c}\in W$ such that 
$u_0vw_0=\ldots=u_{c}vw_{c}$. 

  In Section \ref{S: reduced and periodich}, we discussed assumptions on the equations $u_0vw_0=\ldots=u_{c}vw_{c}$ that imply the periodicity of $v$.  
 In order to prove the lemma we want to apply Proposition \ref{P: periodic} to $U_1vW$ and to $W^{-1}v^{-1}U_1^{-1}$. 

We verify the assumptions of the proposition. We recall that for all $u\in U_1$ and all $v\in U_2$, 
$$2500 \kappa_0 \leqslant |ux_0-x_0|\leqslant |vx_0-x_0|.$$ 

\begin{lem}  \label{L: symmetry}
Let $n>1$. If there are $u_0,\ldots \, u_n \in U_1$,  $w_0,\ldots \, w_n \in W$ and $v\in U_2$ such that 
$$u_0vw_0=\ldots = u_nvw_n, $$ 
then for all $1\leqslant i,j \leqslant n$, 
$ |u_ix_0-u_{j}x_0| \leqslant |vx_0-x_0| \hbox{ and } |w_i^{-1}x_0-w_{j}^{-1}x_0| \leqslant |vx_0-x_0|. $
\end{lem}

\begin{proof} Without restriction, $|u_0x_0-x_0|\leqslant |u_1x_0-x_0|\leqslant \ldots \leqslant |u_nx_0-x_0|$. 
 By Proposition \ref{P: reduced products} and Remark \ref{R: inversed triangle inequality}, if $i>j$,
\begin{align}
|vx_0-x_0|  \geqslant |u_ix_0-x_0|\geqslant |u_ix_0-u_jx_0| + |u_jx_0-x_0| - 48 \delta  \label{Eq: inegality symmetry}
\end{align}
As $|u_jx_0-x_0|\geqslant 2500 \kappa_0$,  $|u_ix_0-u_jx_0| \leqslant  |vx_0-x_0|$, which is the first assertion of the lemma. 

Moreover, by Proposition \ref{P: reduced products}, $x_0,vx_0\in C_{u_i^{-1}u_j}^{+190 \delta}$. Thus, by Lemma \ref{L: axis}, $|u_ix_0-u_jx_0|\geqslant [u_i^{-1}u_j]\geqslant |u_ivx_0-u_jvx_0| - 390\delta.$
Let $h:=u_ivw_i=u_jvw_j$. Then  $h^{-1}u_ivx_0=w_i^{-1}x_0$ and $h^{-1}u_jvx_0=w_j^{-1}x_0$. We conclude that $|u_ivx_0-u_jv x_0| =|w_i^{-1}x_0 - w_j^{-1}x_0|.$ Thus, using \eqref{Eq: inegality symmetry}, 
\begin{align*}
|vx_0-x_0| \geqslant |w_i^{-1}x_0 - w_j^{-1}x_0| + |u_jx_0-x_0| - 438\delta  
\end{align*}
As $|u_jx_0-x_0|\geqslant 2500 \kappa_0$, $ |w_j^{-1}x_0-w_i^{-1}x_0|\leqslant |vx_0-x_0|$.
\end{proof}

\begin{proof}[Proof of Lemma \ref{ML: counting}]
If $|U_1vW|\leqslant \frac1{2c}  |U_1||W|$, then there are distinct $u_0,\ldots,u_{c}\in U_1$ and $w_0,$ $\ldots,$ $w_{c}\in W$ such that 
$u_0vw_0=\ldots=u_{c}vw_{c}$. As periods are unique, see Lemma \ref{L: uniqueness of periods}, the claim follows.
\end{proof}

\subsubsection{Reduction to the case of bi-periodic sets} 

Recall from Definition \ref{D: bi-periodic set} that $U_2\subset G$ is bi-periodic at $x_0$, if there are maximal loxodromic subgroups $E_1$ and $E_2$ such that for all $v\in U_2$ the isometry $v$ is $E_1$--periodic at $x_0$ and $v^{-1}$ is $E_2$--periodic at $x_0$. 

 Recall that the products in $U_1U_2$ and $U_2W$ are reduced at $x_0$  and that for all $u\in U_1$ and all $v\in U_2$, $2500 \kappa_0 \leqslant |ux_0-x_0|\leqslant |vx_0-x_0|$. 
 
\begin{prop}[Bi-periodic sets] \label{P: periods in product sets} Let $c\geqslant 10^{12}N_0^4 \frac{\kappa_0^2}{\rho_0^2}$.  If $|U_1|> 8 c$ and  if $|W|> 8 c$,
 then 
$$ |U_1U_2W|>\frac{|U_1||W|}{8c} \hbox{ unless $U_2$ is bi-periodic at $x_0$.}$$
\end{prop}

The proposition extends \cite[Lemma 2]{safin_powers_2011}. In the remainder of this section, we mimic the argument of Safin \cite[Section 4]{safin_powers_2011} to prove the proposition, keeping his notation.

We keep the notation of Proposition \ref{P: periods in product sets} and assume that $$ |U_1U_2W|\leqslant \frac{|U_1||W|}{8c}$$
so that by Lemma \ref{ML: counting} all elements in $U_2$ are periodic. We now show that all periods of the elements of $U_2$ are the same.  By applying the same arguments to $U_2^{-1}$, this then implies that $U_2$ is a bi-periodic set.

Let $v$, $v'$ in $U_2$  and let  $E$ and $E'$ be maximal loxodromic subgroups such that $v$ is $E$--periodic at $x_0$ and $v'$ is $E'$--periodic at $x_0$. 
 Let  $$U_{00}:=\{ u\in U_1 \mid \hbox{$u$ is $E$-- and $E'$--right-periodic at $x_0$}\}.$$

\begin{lem}\label{L: counting 2}
If $ |U_1U_2W|\leqslant \frac{1}{8 c}\, |U_1||W|$, then
$|U_{00}|> 4c   \hbox{ and } |U_{00}U_2W| \leqslant \frac1{4c}\;  |U_{00}||W|.$
\end{lem}

\begin{proof} 
Let  
$U_{01}$  be the set of all $u\in U_1$ that are not $E$--right-periodic at $x_0$, and let 
$U_{02}$ be the set of all $u\in U_1$ that are not $E'$--right-periodic at $x_0$. We note that $U_1=U_{00}\cup U_{01} \cup U_{02}$. 
By Lemma \ref{ML: counting}, $$\hbox{$|U_{01}|\leqslant 2c$ or $|U_{01}vW|> \frac1{2c} \; |U_{01}||W|$.}$$ 
On the other hand, $|U_1|> 8c \hbox{ and } |U_{01}vW|\leqslant |U_{01}U_2W| \leqslant \frac{1}{8c}\,|U_1||W|.$ Thus   
$|U_{01}|\leqslant \frac14\, |U_1|$. Analogously $|U_{02}|\leqslant \frac14\, |U_1|$. Then $|U_{00}| \geqslant \frac12\, |U_1|.$  Therefore, $|U_{00}| > 4c$  and $ |U_{00}U_2W|\leqslant \frac1{4c}\; |U_{00}||W|.$
\end{proof}

\begin{proof}[Proof of Proposition \ref{P: periods in product sets}] We assume that $|U_1U_2W|\leqslant \frac{|U_1||W|}{4c},$ that $E\not=E'$ and that $[E]\geqslant [E']$. Lemma \ref{L: counting 2}  allows to apply Lemma \ref{ML: counting} again, now using $U_{00}$ instead of $U_1$. 
Thus, the $E'$--right-periods of at least two elements of $U_{00}$ are of Hausdorff distance at least $250 \delta$. This contradicts Lemma \ref{L: right-period}. We conclude that $E=E'$, and, hence, that the period of all $v\in U_2$ is uniquely determined. 

We then apply these arguments to $W^{-1}$, $U_2^{-1}$ and $U_1^{-1}$ at $x_0$. We conclude that the period of all $v^{-1}$ at $x_0$, where $v\in U_2$, is uniquely determined. Thus, $U_2$ is bi-periodic at $x_0$.
\end{proof}

\subsection{Growth of sets of diffuse energy.}
Recall that $U$ is a set of diffuse energy. 
We now combine the reduction lemmas of Section \ref{S: Energy and reduced products} with Proposition \ref{P: periods in product sets} to estimate the growth of $U$. 

If $G$ is hyperbolic and acts properly and cocompactly on a $\delta$--hyperbolic space,  let  $b=10|B(x_0,1000\delta)|$ and recall that $d=1$. 
 If the group $G$ acts on a tree,  $b=10$ and $d=1$. Finally, if $G$ acts acylindrically on $X$, then  $b =10$ and $d=\log_2(2|U|)$.

\begin{lem}\label{L: median} There are subsets $U_1\subset U$  and $U_2\subset U$ of cardinalities $>\frac{1}{2b^2}|U|$ such that for all $u_1\in U_1$ and $u_2\in U_2$ 
$$(u_1^{-1}x_0,u_2x_0)_{x_0}\leqslant 1000d\delta \hbox{ and } (u_2^{-1}x_0,u_1x_0)_{x_0}\leqslant 1000d\delta$$ 
and 
$$10^{10}d\kappa_0 \leqslant |u_1x_0-x_0|\leqslant |u_2x_0-x_0|.$$
\end{lem} 
 \begin{proof} Let $U_1$ and $U_2$ be the sets given by the respective version of the Reduction lemma, Lemma \ref{L: reduction} (hyperbolic groups), Lemma \ref{L: reduction trees} (trees) or Lemma \ref{L: reduction acylindrical} (general acylindrical actions). 
Let $m_1$ be the median of $\{|u_1x_0-x_0| \mid u_1\in U_1 \}$ and let $m_2$ be the median of $\{|u_2x_0-x_0| \mid u_2\in U_2 \}$. 
 If $m_1\leqslant m_2$, let $U_1':=\{u \in U_1 \mid |ux_0-x_0|\leqslant m_1\}$ and $U_2':=\{u \in U_2 \mid |ux_0-x_0|\geqslant m_2\}$. Otherwise, if $m_2 < m_1$, let $U_1':=\{u \in U_1 \mid |ux_0-x_0|\geqslant m_1\}$ and $U_2':=\{u \in U_1 \mid |ux_0-x_0|\leqslant m_1\}$. Finally, redefine $U_1:=U_1'$ and $U_2:=U_2'$.
 \end{proof}

Let us fix subsets $U_1$ and $U_2$ given by Lemma \ref{L: median}. 
Fix $n\geqslant 3$, let $l:=[(n+1)/2]$ and let $W_l:=(U_1U_2)^{l-2}U_1$. Then $|U^n|\geqslant |U_1U_2W_l|$. 
\begin{lem}\label{L: the set W} For all $u_2\in U_2$ and all $w_l\in W_l$, $(u_2^{-1}x_0,w_lx_0)_{x_0} \leqslant 1010 d \delta.$
\end{lem}
\begin{proof}
Let $1\leqslant i\leqslant l$ and $a_i\in U_1$ and $b_i\in U_2$ and let $w_l=a_2b_2\cdots a_{l-1}b_{l}$. Applying Lemma \ref{L: discrete quasi-geodesics}(2) to the sequence of points  $x_1:=a_1x_0$, $x_2:=a_1b_1x_0$, $\ldots$, $x_{2l}:=a_1b_1w_lx_0$, yields that $d(x_2,[x_1,x_{2l-1}])\leqslant 1010d\delta$. Thus $(x_1,x_{2l})_{x_2}=(b_1^{-1}x_0,w_lx_0)_{x_0}\leqslant 1010 d\delta  $ by Lemma \ref{L: Gromov product}. 
\end{proof}

We recall $X$ is $\delta$--hyperbolic, and that the group $G$ acts $(\kappa_0,N_0)$--acylindrically on $X$. Moreover, if $\delta>0$ then $\rho_0=\delta/N_0$, and if $\delta=0$ then $X$ is a simplicial tree of edge length $\rho_0$.
Let  
$c = 10^{12}N_0^4\frac{\kappa_0^2}{\rho_0^2}.$

\begin{prop} \label{P: general product set growth}  Let $U\subset G$ be a finite set of diffuse energy. Then  
 for all positive natural numbers $n$,
$$ |U^n|\geqslant \left(\frac{1}{8}\, \frac{1}{10^{36}d^6} \frac{1}{c} \, \frac{1}{2b^2} |U|\right)^{[(n+1)/2]}.$$ 
\end{prop}
\begin{rem} If $\delta=0$, then $ |U^n|\geqslant \left(\frac{1}{8}\, \frac{1}{c} \, \frac{1}{200} |U|\right)^{[(n+1)/2]}.$
\end{rem}

\begin{proof} Fix $n\geqslant 3$ and let $\delta':=1010 d \delta$. Then $X$ is $\delta'$--hyperbolic and all the all the products in $U_1U_2$ and $U_2W_l$ are reduced at $x_0$. Moreover, the action of $G$ is $(\kappa_0':=\kappa(1010 d),N_0':=N(1010d))$--acylindrical,  see Lemma \ref{L: acylindrical}, and $10^4\kappa_0'\leqslant 10^{10}d\kappa_0$,  $N_0'\leqslant 10^5 d N_0$. Then $c':= 10^{12}{N'_0}^4\frac{{\kappa'_0}^2}{{\rho'_0}^2}\leqslant 10^{36} d^6 c$. 

 If $|U_1|\leqslant 8c'$ or if $|W_l| \leqslant 8 c'$, then, obviously, for every $v\in U_2$, $|U_1U_2W_l|\geqslant|U_1vW_l|\geqslant \frac{|U_1||W_l|}{8c'}$.

Otherwise, by Proposition \ref{P: periods in product sets}
$$ |U_1U_2W_l|>\frac{1}{8c'} {|U_1||W_l|}\hbox{ unless $U_2$ is bi-periodic at $x$.}$$ 
If $U_2$ is not bi-periodic at $x_0$, we conclude by induction that for all positive $n$,
$$ |U^n|\geqslant \left(\frac{1}{8c'}|U_1| \right)^{[(n+1)/2]}\geqslant \left(\frac{1}{8}\, \frac{1}{10^{36}d^6}\,   \frac{1}{c} \, \frac{1}{2b^2}  |U|\right)^{[(n+1)/2]}.$$
 Otherwise, we conclude with Proposition \ref{P: product growth bi-periodic sets}.
\end{proof}
 
 This concludes the proof of Theorems \ref{IT: hyperbolic groups}, \ref{IT: acylindrical trees} and \ref{IT: acylindrical hyperbolic}.


\addtocontents{toc}{\setcounter{tocdepth}{-10}}
\bibliographystyle{alpha}

\bibliography{product}

\end{document}